\documentclass[11pt]{article}
\usepackage{amsthm}
\usepackage{amssymb,amsmath,amsfonts,amssymb}
\usepackage{graphics,graphicx,color}
\usepackage{caption}
\usepackage{subcaption}
\usepackage{fullpage}

\newtheorem{theorem}{Theorem}
\newtheorem{lemma}{Lemma}
\newtheorem{proposition}{Proposition}

\newcommand\rvec{{\bf r}}

\newcommand\nvec{{\bf n}}
\newcommand\xvec{{\bf x}}
\newcommand\rhat{{\bf \hat r}}

\newcommand\pvec{{\bf p}}

\newcommand\Qvec{{\bf Q}}
\newcommand\mvec{{\bf m}}
\newcommand\eps{{\epsilon}}

\newcommand{\Rr}{{\mathbb R}}

\title{Front Propagation at the Nematic-Isotropic Transition Temperature}
\author{Apala Majumdar, Paul A. Milewski and Amy Spicer}

\begin{document}
\maketitle
\begin{abstract}
 We study the gradient flow model for the  Landau-de Gennes energy functional for nematic liquid crystals at the nematic-isotropic transition temperature on
prototype geometries. We study the dynamic model on a three-dimensional droplet and on a disc with Dirichlet boundary conditions and different types of initial conditions. In the case of a droplet with radial boundary conditions, a large class of physically relevant initial conditions generate dynamic solutions with a well-defined isotropic-nematic interface which propagates according to mean curvature for small times. On a disc, we make a distinction between ``planar" and ``non-planar" initial conditions and  ``minimal" and ``non-minimal" Dirichlet boundary conditions. Planar initial conditions generate solutions with an isotropic core for all times whereas non-planar initial conditions generate solutions which escape into the third dimension. Non-minimal boundary conditions generate solutions with boundary layers and these solutions can either have a largely ordered interior profile or an almost entirely disordered isotropic interior profile. Our examples suggest that whilst critical points of the Landau-de Gennes energy typically have highly localized disordered-ordered interfaces, the transient dynamics exhibit observable isotropic-nematic interfaces of potential experimental relevance.%
 \end{abstract}

\section{Introduction}
\label{sec:intro}
Nematic liquid crystals are classical examples of mesophases between conventional solid and
liquid phases; they are anisotropic liquids with preferred directions of molecular alignment, these directions being referred to as "directors" in the literature \cite{dg,virga}. In other words, nematic liquid crystals are complex liquids with orientational order. Nematics in confinement are an exciting source of challenges for mathematicians and practical scientists alike. We study dynamically metastable nematic
configurations with interfaces in prototype geometries within the Landau-de Gennes (LdG) theory for nematic liquid crystals. In this framework, the nematic
state is described by the $\Qvec=\left\{Q_{ij}\right\}$-tensor, a symmetric, traceless $3\times 3$ matrix, interpreted as a macroscopic measure
of the nematic order \cite{dg,virga}. The LdG energy functional
comprises a bulk potential, determining nematic
order as a function of temperature, and an elastic energy
density which penalizes spatial inhomogeneities. We work at the
nematic-isotropic transition temperature, where both the isotropic and nematic phases are minimizers of the bulk potential,
and with the one-constant elastic energy density in
the limit of vanishing elastic constant. This limit, describing macroscopic domains (see
\cite{cpaa2010}) with length scales much larger than the nematic correlation length, is studied in detail in the context of energy minimizers in \cite{amaz}.
 
We adopt the gradient flow model to describe the nematodynamics in the absence of fluid flow at a constant temperature. Gradient flows are evolution equations driven by a decreasing energy \cite{gradientflow}. Our model is derived from the LdG energy with the $L^2$-norm as the dissipation mechanism and is described by a system of five coupled
nonlinear parabolic partial differential equations. Standard theory for
parabolic systems shows that we have a unique ``dynamic" time-dependent solution for physically relevant initial and boundary conditions. 
Gradient flows have been used in the context of LC dynamics. For example, in \cite{popa1, popa2} the authors study a one-dimensional gradient flow model and the effects of biaxiality and elastic anisotropy.  In \cite{weiwang}, the authors study isotropic-nematic front propagation using the method of matched asymptotic expansions within the more general Beris-Edwards theory for nematodynamics \cite{sengupta}. In particular, they account for fluid flow and the coupling between fluid flow and nematic order. They derive evolution laws for the velocity field, the director field
 of nematic alignment and the isotropic-nematic interface but without any special attention to the effects of boundary conditions and initial conditions. We work in a simpler dynamical framework with no fluid flow but with focus on how the dynamics is affected by the choice of boundary and initial data.

At the nematic-isotropic transition temperature, the LdG bulk potential bears strong resemblance to the Ginzburg-Landau (GL) potential in superconductivity. In our first model problem we study a three-dimensional droplet with Dirichlet radial boundary conditions.
We use the concept of ``normalized energy" for the GL gradient flow model in  \cite{bronsardkohn, bronsardstoth}, to prove that isotropic-nematic interfaces propagate according to mean curvature  in certain model situations. The long-time dynamics is described by an explicit critical point of the LdG energy - the radial hedgehog (RH) solution \cite{maj2012, gartlandmkaddem2002, henaomajumdar2012}. The RH solution has perfect radial symmetry, with perfect radial nematic alignment and an isolated isotropic point at the centre, referred to as a point defect in the literature
\cite{gartlandmkaddem2002}. We focus on the interplay between initial conditions and transient dynamics followed by convergence to the static RH solution using four different representative initial
conditions.
The transient dynamics has some universal features
which may have experimental repercussions. 
 
In Section~\ref{sec:cylinder}, our second model problem focuses on dynamic solutions on a disc with Dirichlet conditions, subject to two distinct types of initial conditions: planar and
non-planar. Planar $\Qvec = \left\{Q_{ij}\right\}$-tensors, 
$i,j =1,2,3$, have zero $Q_{13}$ and $Q_{23}$ components and non-planar $\Qvec$-tensors do not. Using standard techniques, we prove that planar
initial conditions evolve to planar dynamic solutions which have an isotropic point at the centre of the disc for all times. These solutions develop an isotropic-nematic interface which propagates inwards and is arrested at the origin. Non-planar initial conditions, including small ``non-planar"
perturbations of planar initial conditions, converge to a
universal non-planar profile. The ``small-time" dynamics are
almost indistinguishable from the planar case;
however, the interface collapses at the origin and the dynamic solution escapes into an entirely ordered non-planar state. 
We track the transient dynamics and numerically compute quantitative estimates for the ``persistence time" of the interface. 

These numerical results are complemented by some analysis for radially symmetric planar critical points of the LdG energy that have been reported in \cite{fritta} for low temperatures. We generalize some of the results in \cite{fritta} to the nematic-isotropic transition temperature and use these critical points to construct radially symmetric and non-symmetric initial conditions for the numerical simulations. The different types of initial conditions suggest that the transient dynamics have universal features independent of the symmetry or uniaxiality/biaxiality of the initial condition.  Namely, in all cases, we have a well-defined isotropic-nematic interface as a pronounced feature of the evolution trajectory, which we illustrate by the eigenvalue evolution of the corresponding LdG $\Qvec$-tensor and plots of $|\Qvec|^2$ as a function of time.


The first two model problems have minimal boundary conditions which are minimizers of the LdG bulk potential. In the last section, we study two-dimensional (2D) and three-dimensional (3D) LdG dynamic solutions on a disc with non-minimal boundary conditions. The two-dimensional case can be easily understood and all dynamic solutions exhibit a rapidly growing 
isotropic core with a thin boundary layer near the lateral surface. The three-dimensional solution landscape is richer and the transient dynamics is sensitive to the initial condition. If the initial condition is planar with an isotropic-nematic interface relatively close to the centre of the disc, then the interface propagates towards the centre, replicating the planar dynamics in Section~\ref{sec:cylinder}. If the planar initial condition has an isotropic-nematic interface relatively close to the boundary of the disc, the interface propagates outwards, yielding an almost entirely isotropic interior and replicating the two-dimensional dynamics. In all cases, we have a boundary layer to match the fixed non-minimal Dirichlet condition. Here, the transient dynamics is sensitive to the initial interface location, a feature which is missing in the model problems with minimal Dirichlet boundary conditions.

The paper is organized as follows. In Section~\ref{sec:prelim}, we
present the gradient flow model for the Landau-de Gennes energy.
In Section~\ref{sec:RH}, we study dynamic solutions on a
droplet with Dirichlet radial conditions.
Section~\ref{sec:cylinder} follows with emphasis on planar and non-planar initial
conditions on a disc and Section~\ref{sec:2D} illustrates the diverse possibilities with non-minimal boundary conditions. We conclude in
Section~\ref{sec:conclusions} with future perspectives.

\section{Preliminaries}
\label{sec:prelim}

The LdG $\Qvec$-tensor order parameter is in the space of symmetric traceless $3\times 3$ matrices, $S_0 =\left\{ \Qvec\in \mathbb{M}^{3\times 3}:
Q_{ij} = Q_{ji}, Q_{ii} = 0 \right\}$. A $\Qvec$-tensor is
said to be (i) isotropic if $\Qvec=0$, (ii) uniaxial if $\Qvec$
has a pair of degenerate non-zero eigenvalues and (iii) biaxial if
$\Qvec$ has three distinct eigenvalues \cite{dg,newtonmottram}. A
uniaxial $\Qvec$-tensor can be written as $\Qvec_u = s \left(\nvec \otimes \nvec -\mathrm{I}/3\right)$ with $s \in \Rr$ and $\nvec\in
S^2$, a unit vector. The scalar, $s$, is an order parameter which measures the degree of orientational order. The vector, $\nvec$, is
referred to as the ``director" and labels the single
distinguished direction of uniaxial nematic alignment \cite{virga,dg}. 


We work with a simple form of the LdG energy given by
\begin{equation}
\label{eq:2} I[\Qvec] =  \int_{\Omega} \frac{L}{2} \left| \nabla
\Qvec \right|^2 + f_B(\Qvec)~ \mathrm{d}V
\end{equation}

where
\begin{eqnarray}
\label{eq:3}
|\nabla \Qvec |^2 = \frac{\partial \Qvec_{ij}}{\partial \xvec_k}\frac{\partial \Qvec_{ij}}{\partial \xvec_k}, \quad f_B(\Qvec) = \frac{A}{2} \textrm{tr}\Qvec^2 - \frac{B}{3}
\textrm{tr}\Qvec^3 + \frac{C}{4}\left(\textrm{tr}\Qvec^2 \right)^2.
\end{eqnarray}

The variable $A = \alpha (T - T^*)$ is the re-scaled temperature, $\alpha, L, B, C>0$ are material-dependent constants and $T^*$ is the characteristic nematic supercooling temperature \cite{dg,newtonmottram}.
Further $\textrm{tr}\Qvec^2 = Q_{ij}Q_{ij}$ and
$\textrm{tr}\Qvec^3 = Q_{ij} Q_{jk}Q_{ki}$ for $i,j,k=1,2,3$. It is well-known that all stationary points of the thermotropic
potential, $f_B$, are either uniaxial or isotropic \cite{dg,newtonmottram,ejam2010}.
The re-scaled temperature $A$ has three characteristic values: (i) $A=0$, below which the isotropic phase $\Qvec=0$ loses stability, (ii) the nematic-isotropic transition temperature,
$A={B^2}/{27 C}$, at which $f_B$ is minimized by the isotropic
phase and a continuum of uniaxial states with $s=s_+={B}/{3C}$ and
$\nvec $ arbitrary, and (iii) the nematic supercooling temperature, $A = {B^2}/{24 C}$, above which the ordered nematic equilibria do not exist.

Throughout this paper we work at the nematic-isotropic transition temperature, investigating the propagation of fronts
separating the isotropic phase from the ordered equilibria 
in the limit, ${L C}/{R^2 B^2} \to
0^+$, where $R$ is a characteristic length scale of the domain $\Omega$. We refer to this as the \emph{vanishing elastic constant limit} for fixed values of $R, B, C$, by analogy with the terminology in \cite{amaz}.
Continuum formulations are typically valid in this limit \cite{maj2012, cpaa2010}.
We work with the gradient flow
model associated with the LdG energy \cite{pre2012} and the dynamic equations are given by:
\begin{equation}
\label{eq:6}\gamma \Qvec_t = L \Delta \Qvec - A\Qvec + B\left(\Qvec
\Qvec - \frac{\mathbf{I}}{3}|\Qvec|^2 \right)  - C|\Qvec|^2 \Qvec,
\end{equation}

where $\gamma$ is a positive rotational viscosity, $\Qvec \Qvec = \Qvec_{ij}\Qvec_{jk}$ with $i,j,k=1,2,3$
and $\mathbf{I}$ is the $3\times 3$ identity matrix. The system
(\ref{eq:6}) comprises five coupled nonlinear parabolic partial
differential equations. We recall a basic result about the
existence and uniqueness of solutions for such gradient flow systems:

\begin{proposition}
\label{prop:1} Let $\Omega \subset \Rr^3$ be a bounded domain with
smooth boundary, $\partial \Omega$. Given a smooth fixed boundary 
condition $\Qvec(\xvec,t) = \Qvec_b(r)$ on $\partial \Omega$ and
smooth initial condition, $\Qvec(\xvec,0) = \Qvec_0(\xvec)$, the
parabolic system (\ref{eq:6}) has a unique solution,
$\Qvec(\xvec,t) \in C^{\infty}(\Omega)$ for all $t>0$.
\end{proposition}

\begin{proof}
The existence of a solution is standard; see \cite{existence} for a proof. From \cite{MajumdarLiquidCrystal, amaz}, we
have the dynamic solution is bounded for all times with
$\left| \Qvec \left(\rvec, t \right) \right|
 \leq \sqrt{2/3}\,B/3C$ for $t \geq 0$. The uniqueness result follows from an immediate application of
Gronwall's inequality to the difference $\Qvec_d = \Qvec_1 - \Qvec_2$ of two solutions, $\Qvec_1$ and $\Qvec_2$,
subject to the fixed boundary condition and the same initial
condition. In particular, $\Qvec_d(\xvec,t) = 0$ on $\partial
\Omega$ and $\Qvec_d(\xvec, 0) = 0$ for $\xvec \in \Omega$. One
can then show that $\Qvec_d(\xvec,t) = 0$ for $\xvec\in\Omega$ and
for all $t>0$. 
\end{proof}

\section{Front propagation on three-dimensional spherical
droplets}
\label{sec:RH}
Our first example concerns nematic droplets. Let $\Omega$ be the unit ball in three dimensions;
$\Omega:=\left\{ \xvec \in \Rr^3;|\xvec| \leq 1 \right\}$. We work with a uniaxial Dirichlet boundary condition 
\begin{equation}
\label{Qb}
\Qvec_b = \frac{B}{3C}\left(\rhat\otimes \rhat -
\frac{\mathbf{I}}{3} \right),
\end{equation}

where $\rhat$ is the 3D radial unit vector. $\Qvec_b$ is a minimizer of the bulk potential $f_B$
in (\ref{eq:3}). For illustration, we first work with uniaxial radial initial conditions
that have a front structure such as
\begin{equation}
\label{eq:IC}
 \Qvec(\xvec,0) =
\begin{cases} \mathbf{0} & 0< |\xvec|< r_0
\\ \Qvec_b & r_0<|\xvec| \leq 1,
\end{cases}
\end{equation} 

for some $\frac{1}{2}\leq r_0 <1$. We refer to these as
``radial hedgehog'' type initial conditions by analogy with the static radial hedgehog solution as described in Section~\ref{sec:intro} \cite{maj2012,gartlandmkaddem2002,henaomajumdar2012}.
We are interested in the qualitative properties of dynamic
solutions of (\ref{eq:6}) subject to these initial and boundary conditions, such as front  propagation and transient dynamics. Looking for ``dynamic" radial hedgehog type solutions,
we work with an ansatz of the form
\begin{equation}
\label{eq:7} \Qvec(\rvec,t)= h(r,t)\left(\rhat\otimes \rhat - \frac{\mathbf{I}}{3}
\right),
\end{equation}

where $h:\left[0, 1 \right]\times \left[0,\infty \right) \to \Rr$
is the scalar order parameter that only depends on $r$,
the radial distance from the origin, and time. On substitution into (\ref{eq:6}), we have a solution of the form (\ref{eq:7}) if the scalar order parameter $h$ is a solution of
\begin{eqnarray}
\label{eq:8} \gamma h_t = h_{rr}+ \frac{2}{r} h_r - \frac{6h}{r^2} +
\frac{3}{\bar{L}}h \left(h_+ - h \right)\left( 2h - h_+ \right),
\end{eqnarray} 
with $h_+ = \frac{B}{3C}$ and $\bar{L} = \frac{9 L}{C}$. The boundary conditions are
$h(0,t) = 0$, $h(1,t) = h_+$ for all $t\geq 0$ and the initial condition is
\begin{equation}
\label{eq:10} h(r, 0) = \begin{cases} 0 & 0\leq r < r_0 \\
h_+ & r_0 < r< 1.
\end{cases}
\end{equation} The evolution equation for $h$ in
(\ref{eq:8}) is simply the gradient flow model
associated with the energy functional
\begin{equation}
\label{eq:11} \frac{I [h]}{4 \pi \sqrt{\bar{L}}} = \int_{0}^{1}
r^2 \sqrt{\bar{L}}\left\{ \frac{1}{3} \left(\frac{dh}{dr}\right)^2
+ \frac{2h^2}{r^2} + \frac{h^2 \left( h - h_+^2
\right)^2}{\sqrt{\bar{L}}} \right\}~ \mathrm{d}r.
\end{equation}

Given a smooth solution $h(r,t) \in C^{\infty}([0,1]\times[0,\infty))$, we can appeal to the uniqueness result in Proposition~\ref{prop:1} to deduce that (\ref{eq:7}) is the physically relevant solution
and hence, the five-dimensional evolution problem in (\ref{eq:6}) reduces to a single evolution equation for a scalar order parameter.

In \cite{bronsardkohn}, the authors study a closely related problem for front propagation on three-dimensional balls in the Ginzburg-Landau framework. They rigorously prove that for suitably defined initial conditions (as in (\ref{eq:10})) with appropriately bounded energy, the front propagates according to mean curvature. Our governing equation (\ref{eq:8}) is similar to that studied in Section~$3$ of \cite{bronsardkohn}, however we have an extra term: $-{6h}/{r^2}$ in (\ref{eq:8}).  In particular, we cannot quote results from \cite{bronsardkohn} and \cite{bronsardstoth} without verifying that the key inequalities are unchanged by the additional term for $\bar{L}$ sufficiently small. In the next paragraphs, we verify the necessary details to reach the desired conclusion.
Let $\rho(t)$ be the solution to
\begin{eqnarray}
\label{eq:12} \frac{d \rho}{dt} = - \frac{2}{\rho}, \quad \rho(0) =
r_0 \in \left(1/2, 1 \right),
\end{eqnarray} 
or alternatively, $\rho(t) = \sqrt{r_0^2 - 4t }$. We define $ T_1 = \frac{1}{4}\left( r_0^2 - \frac{1}{4} \right)$. This is the first time for which $\rho(t) = {1}/{2}$ and is independent of $\bar{L}$. Next, let 
\begin{eqnarray}
\label{eq:14} f(r,t) = \begin{cases} 0  & 0\leq r \leq \rho(t) \\
h_+ &\rho(t) \leq r \leq 1. \end{cases}
\end{eqnarray}
Our goal is to show that the solution $h(r,t)$ of (\ref{eq:8}), subject to suitably defined initial conditions, resembles the function $f(r,t)$ for $T < T_1$, in the sense that %
\begin{equation}
\label{eq:15} \int_{0}^{1} r^2 \left| h(r, t) - f(r,t) \right|~ \mathrm{d}r
\to 0 ~\textrm{as $\bar{L} \to 0$}.
\end{equation} 

As in \cite{bronsardkohn}, the key step is to define a weighted energy as shown below:
\begin{eqnarray}
E_\phi [w](\tau) = \int_{-\rho(\tau)}^{1 -
\rho(\tau)} \phi(R, \tau) \left\{ \sqrt{\bar{L}}\left(
\frac{w_R^2}{3} + \frac{ 2 w^2}{(R + \rho)^2} \right) + \frac{w^2
\left(h_+ - w \right)^2}{\sqrt{\bar{L}}} \right\}~\mathrm{d}R, \nonumber
\end{eqnarray}
where
\begin{equation}
\label{eq:17} w(R, \tau) = h(R + \rho(t), t ), \quad -\rho(\tau)
\leq R \leq 1 - \rho(\tau); \quad \tau\geq 0,
\end{equation} 

and $\phi(R, \tau)$ is a weight function
\begin{equation}
\label{eq:20} 
\phi( R, \tau) = \textrm{exp}\left[ - \frac{
2R}{\rho} \right] \left( 1 + \frac{R}{\rho} \right)^2.
\end{equation} 

In particular, $w\left(R, \tau \right)$ is a solution of
\begin{eqnarray}
\label{eq:23} \sqrt{\bar{L}}w_\tau - \frac{\sqrt{\bar{L}}}{\phi}
\left( \phi w_R \right)_R + \frac{ 6 w  \sqrt{  \bar{L}}}{(R +
\rho )^2} + \frac{3}{\sqrt{\bar{L}}}w(h_+ - w)(2w - h_+) = 0.
\end{eqnarray} 

We follow the steps in Proposition $3.2$ of \cite{bronsardkohn} to show that
\begin{equation}
\label{eq:29}
\frac{d}{d\tau} E_\phi [w](\tau)  \leq -\frac{2}{3}\frac{\sqrt{\bar{L}}}{\gamma}\int_{-\rho(\tau)}^{1 - \rho(\tau)} \phi(R,\tau) w_\tau^2~ \mathrm{d}R + \frac{ 4 \phi}{\rho}\bar{L}^{1/2} h_+^2,
\end{equation}
for $\tau \leq T_1$.
By contrast, in \cite{bronsardkohn}, the weighted energy in the GL-framework is strictly decreasing. We have lesser control on the weighted energy  but for $\bar{L}$ sufficiently small, the inequality (\ref{eq:29}) suffices for our purposes.
 
 Next, we define an interface energy which yields lower bounds for the weighted energy:
  \begin{equation}
 \label{eq:31}
 g(s) = \frac{2}{\sqrt{3}}
 \int_{0}^{s} w  \left( h_+ - w \right)~\mathrm{d}w,
 \end{equation}
and so $g(h_+) = \frac{h_+^3}{3 \sqrt{3}}$ is an \emph{interface energy} associated with an isotropic-nematic front. Further, let
\begin{eqnarray}
 \label{eq:32}
 v(R) = \begin{cases} 0 &-\rho(\tau) \leq R < 0 \\
 h_+ & 0  < R < 1 - \rho(\tau). \end{cases}
 \end{eqnarray}
 We can adapt a lemma in \cite{bronsardstoth} to show that:\\
 
\begin{proposition}
\label{prop:lb}
If for some smooth function $w$,
\begin{eqnarray}
  \label{eq:33}
 \int_{-a}^{a} \left| g(w) - g (v) \right|~\mathrm{d}s \leq \frac{g(h_+)}{4}\bar{L}^{\alpha} \quad
 \textrm{and} \quad E_\phi [w] \leq C_1 ,
 \end{eqnarray}
  where $0<\alpha<{1}/{4}$ and $a  =\rho(T_1)/2\sqrt{2}$, then
 $E_\phi[w](\tau) \geq g(h_+) - C_2 \bar{L}^{1/2-\alpha} - C_3 \bar{L}^{2\alpha}$
 for $\tau\leq T_1$ and positive constants $C_1, C_2, C_3$ independent of $\bar{L}$.
\end{proposition}
It remains to construct initial conditions $w(R,0)$, which satisfy the hypothesis of Proposition~\ref{prop:lb}. The construction is parallel to that in Equation $(1.22)$ of \cite{sternberg} and we give a statement for completeness.\\
\begin{proposition}
\label{prop:IC}
Define the function
\begin{equation}
\label{eq:sigma}
\sigma(R) = \frac{h_+}{1 + \exp[ -\sqrt{3}h_+ R ]}.
\end{equation} At $\tau=0$, $R= r- r_0$ and for $\bar{L}$ sufficiently small, define
\begin{equation}
\label{eq:IC1}
w\left( R \right) = \begin{cases} h_+ & R > 2 \bar{L}^{1/4}\\[10pt]
\left(h_+ - \sigma\left(\frac{1}{\bar{L}^{1/4}}\right)\right) \left( R - 2 \bar{L}^{1/4} \right) /\bar{L}^{1/4}+ h_+ & \bar{L}^{1/4}\leq R \leq 2\bar{L}^{1/4} \\[12pt]
  \sigma\left(\frac{R}{\sqrt{\bar{L}}}\right) & -\bar{L}^{1/4} \leq R \leq \bar{L}^{1/4} \\[10pt]
\sigma\left(-\frac{1}{\bar{L}^{1/4}}\right)\left( R + 2 \bar{L}^{1/4} \right)
/{\bar{L}^{1/4}} & -2\bar{L}^{1/4} \leq R \leq -\bar{L}^{1/4}\\[12pt]
0 & R < -2\bar{L}^{1/4}.
\end{cases}
\end{equation}
Then
$E_\phi[w] \leq g(h_+) + C \bar{L}^{1/4}$
for a positive constant $C$ independent of $\bar{L}$, as $\bar{L}\to 0$.
\end{proposition}

Finally, we have by analogy with the main theorem in \cite{bronsardstoth}:
\\
\begin{proposition}
\label{prop:3}
Let $0< \alpha \leq \frac{1}{4}$ and assume that $E_\phi[w](0) \leq g(h_+) + c_1 \bar{L}^{2\alpha}$ for some constant $c_1>0$ independent of $\bar{L}$ and that for $a$ as above,
\begin{equation}
\label{eq:55}
\int_{-a}^{a} \left| g\left(w(R, 0) \right) - g(v) \right|~\mathrm{d}R < \frac{g(h_+)}{8} \bar{L}^{\alpha}.
\end{equation}
Let $T_\eps$ be the first time for which
\begin{equation}
\label{eq:56}
\int_{-a}^{a} \left| g(w(R, T_\eps)) - g(w(R,0)) \right|\,\mathrm{d}R = \frac{g(h_+)}{8}\bar{L}^{\alpha}.
\end{equation}
Then $T_\eps \geq \min\left(T_1, C \right)$ for some positive constant $C$ independent of $\bar{L}$ as $\bar{L} \to 0^+$. In other words, we have
\begin{equation}
\label{eq:57}
\int_{-a}^{a} \left| g(w(R, \tau )) - g(v) \right|\,\mathrm{d}R < \frac{g(h_+)}{4}\bar{L}^{\alpha}
\end{equation} 
for all $\tau < T_\eps$ and $T_\eps$ is of order one.
\end{proposition}

The proof follows verbatim from \cite{bronsardstoth}. Equipped with a weighted energy, estimates for the rate of change of the weighted energy and bounds for the weighted energy along with suitable initial conditions, we adapt arguments from Theorem~$3.1$ of \cite{bronsardkohn} to prove:
\\
\begin{theorem}
\label{thm:1}
Let $\Omega$ be the unit ball in $\Rr^3$. Let $h_{\bar{L}}\left(r, t \right)$ denote the solution of the evolution equation (\ref{eq:8}) subject to the fixed boundary conditions and an initial condition with an interface structure and appropriately bounded weighted energy:
\begin{equation}
\label{eq:66}
\int_{0}^{1}\psi_0(r)\left[ \sqrt{\bar{L}}\left(
\frac{w_R^2}{3} + \frac{ 2 w^2}{(R + \rho)^2} \right) + \frac{w^2
\left(h_+ - w \right)^2}{\sqrt{\bar{L}}} \right]r^2\, \mathrm{d}r \leq g(h_+) + \Gamma \bar{L}^{1/4}
\end{equation}
with $\Gamma$ independent of $\bar{L}$ and
$\psi_0\left( r \right) = \frac{1}{r_0^2} \exp\left[-2\left(\frac{r}{r_0} - 1 \right)\right]$. Then for any $T < T_\eps$, where $T_\eps$ has been defined in Proposition~\ref{prop:3}, we have
\begin{equation}
\label{eq:68}
\lim_{\bar{L}\to 0} \int_{0}^{T}\int_{\Omega} \left| h_{\bar{L}}\left(r, t \right) - f\left(r, t\right) \right| r^2\, \mathrm{d}r \mathrm{d}t = 0
\end{equation}
where $f(r,t)$ has been defined in (\ref{eq:14}).
\end{theorem}

\textit{Comments on the proof:} The condition (\ref{eq:66}) is equivalent to the bound in Proposition~\ref{prop:IC}, which can be realized by initial conditions with ``efficient interfaces". The key ingredients are the rate of change of the weighted energy in (\ref{eq:29}), the lower bound for the weighted energy in terms of $g(h_+)$ in Proposition~\ref{prop:lb} and the upper bound in (\ref{eq:66}). These ensure that the system does not have sufficient energy to create additional interfaces away from $R=\rho(\tau)$ and $h$ is effectively constant (either $h=0$ or $h=h_+$ ) away from $R=\rho(\tau)$. Given $h_{\bar{L}}\left( r, t \right)$, we obtain the unique solution $\Qvec_{\bar{L}}\left(\xvec, t \right)$ of the evolution equation (\ref{eq:6}), subject to the boundary and initial conditions, given by $\Qvec_{\bar{L}}\left(\xvec, t \right) =
h_{\bar{L}}\left(|\xvec|, t \right)\left(\rhat \otimes \rhat - \frac{\mathbf{I}}{3} \right)$.

\subsection{Numerical simulations on the sphere}
We numerically compute solutions of the gradient flow system (\ref{eq:6}) on a three-dimensional droplet, with the fixed boundary condition $\Qvec_b$ in \eqref{Qb} and various types of initial conditions. From the numerical results in \cite{gartlandmkaddem2002}, the radial hedgehog (RH) solution is the global minimizer of the LdG energy for this model problem; the RH solution is a uniaxial solution of the form $\mathbf{H} = s(r)\left(\rhat \otimes \rhat - \frac{\mathbf{I}}{3}\right)$ where $s(0)=0$ and $s(r)>0$ for $r>0$. In particular, $s$ rapidly interpolates between $s=0$ and the boundary value of $s_+ ={B}/{3C}$ over a distance proportional to the nematic correlation length, $\xi \propto \sqrt{{L C}/{B^2}}$, and the localized region of reduced order near $r=0$ is referred to as the "defect core". We expect the long-time dynamics to converge to the RH solution for all choices of initial conditions. However, we are equally interested in the transient dynamics and the dynamic persistence of isotropic-nematic interfaces. 
 In what follows, we look at four different initial conditions: uniaxial initial conditions within the remit of Theorem~\ref{thm:1}, uniaxial initial conditions outside the scope of Theorem~\ref{thm:1}, biaxial initial conditions and initial conditions that break the radial symmetry of the order parameter. 

Let $R$ be the radius of the droplet and we non-dimensionalize the system (\ref{eq:6}) by setting
$\bar{t} = \frac{20 t L}{ \gamma R^2},\,\bar{\xvec} = \frac{\xvec}{R}$ to yield
\begin{eqnarray}
&&  \frac{\partial Q_{11}}{\partial \bar{t}}= \bar{\Delta} Q_{11} -\frac{1}{\tilde{L}}\bigg({A}Q_{11}+2{C}(Q_{11}^2+Q_{22}^2 +Q_{12}^2 + Q_{11}Q_{22} + Q_{13}^2 + Q_{23}^2)Q_{11}  \nonumber \\
&&-\frac{{B}}{3}(Q_{11}^2 +Q_{12}^2 +Q_{13}^2-2Q_{22}^2 -2Q_{11}Q_{22}-2Q_{23}^2)\bigg),  \label{eq:n1}\\
&&\frac{\partial Q_{22}}{\partial\bar{t}}=\bar{\Delta} Q_{22} -\frac{1}{\tilde{L}}\bigg(  {A}Q_{22}+2{C}(Q_{11}^2+Q_{22}^2 +Q_{12}^2 + Q_{11}Q_{22} + Q_{13}^2 + Q_{23}^2)Q_{22} \nonumber\\
&&-\frac{{B}}{3}(Q_{12}^2 +Q_{22}^2 +Q_{23}^2-2Q_{11}^2 -2Q_{11}Q_{22}-2Q_{13}^2)\bigg),  \label{eq:n2}\\
&&\frac{\partial Q_{12}}{\partial\bar{t}}=\bar{\Delta} Q_{12} -\frac{1}{\tilde{L}} \bigg(  {A}Q_{12}+2{C}(Q_{11}^2+Q_{22}^2 +Q_{12}^2 + Q_{11}Q_{22} + Q_{13}^2 + Q_{23}^2)Q_{12}  \nonumber\\
&&-{B}(Q_{11}Q_{12}+Q_{12}Q_{22} +Q_{13}Q_{23})\bigg),  \label{eq:n3} \\
&&\frac{\partial Q_{13}}{\partial\bar{t}}=\bar{\Delta} Q_{13} -\frac{1}{\tilde{L}}\bigg( {A}Q_{13}+2{C}(Q_{11}^2+Q_{22}^2 +Q_{12}^2 + Q_{11}Q_{22} + Q_{13}^2 + Q_{23}^2)Q_{13}  \nonumber\\
&&-{B}(Q_{12}Q_{23}-Q_{22}Q_{13})\bigg),  \label{eq:n4}\\
&&\frac{\partial Q_{23}}{\partial \bar{t}}=\bar{\Delta} Q_{23} -\frac{1}{\tilde{L}}\bigg( {A} Q_{23}+2{C}(Q_{11}^2+Q_{22}^2 +Q_{12}^2 + Q_{11}Q_{22} + Q_{13}^2 + Q_{23}^2)Q_{23} \nonumber\\
&&-{B}(Q_{12}Q_{13}-Q_{23}Q_{11})\bigg)  \label{eq:n5},
\end{eqnarray}
where $\bar{\Delta}$ denotes the Laplacian with respect to the re-scaled coordinate $\bar{\xvec}$.
(In what follows, we drop the \emph{bars} from the dimensionless variables).
 We take $R^2 = 10^{-10} \,m^2$, $\tilde{L} = \frac{L}{R^2} N/m^2$, $B=0.64\times 10^4\, N/m^2$, $C=0.35 \times 10^4\, N/m^2$ and $A=\frac{B^2}{27 C}$ throughout the paper and work with either $\tilde{L}=0.05$ or $0.01$ \cite{gartlandmkaddem2002}.\\
Here, and in the subsequent sections of this paper, the system of 
reaction-diffusion equations (\ref{eq:n1} - \ref{eq:n5}) is solved as follows. 
The unit ball is embedded into the unit cube $[-1,1]^3$ which is discretised with a uniform cartesian grid with spatial resolution $h$. 
We implement a special case of an immersed boundary method (see for example, \cite{Zhong:2007aa}), and apply the boundary conditions at all discrete points within distance $h/2$ of the boundary. For interior points, the solution satisfies the system (\ref{eq:n1}-\ref{eq:n5}) and in the exterior of the physical domain, we solve the simple heat equation (i.e. take $A=B=C=0$) and use periodic boundary conditions on the cube. This setup makes it simple and efficient to use higher order and spectral schemes for spatial derivatives. Timestepping is accomplished with a standard fourth-order Runge-Kutta scheme. Simple finite difference schemes are also implemented to verify the results.

The first two initial conditions are uniaxial RH type initial conditions of the form 
\begin{equation}
\label{eq:rh1}
\Qvec(\rvec, 0) = h(r,0)\left(\rhat \otimes \rhat - \frac{\mathbf{I}}{3} \right).
\end{equation} 
\textbf{Case I} prescribes an initial condition, $h(r,0)$, with an interface structure given by
$h(r)=\frac{1}{2}h_+\left(1+\tanh \left(\left(r-r_0\right)/\sqrt{\tilde{L}}\right)\right)$,
 and \textbf{Case II} describes an initial condition without an interface structure with
$h(r)=h_+ r$.  Case I is within the remit of Theorem~\ref{thm:1} and the numerics demonstrate that the solution retains the isotropic-nematic interface for all times and that the interface propagates towards the origin according to mean curvature for small times, equilibriating near the origin for long times. For long times, the radius of the isotropic core scales, as expected, with $\sqrt{\tilde{L}}$ and arises out of the saddle structure of $\boldsymbol{Q}$ at the origin. The dynamic solution in Case II very quickly develops an inwards-propagating interface separating the isotropic core at the centre from the ordered nematic state and then follows the same evolution path as Case I
(see Figure~\ref{InterfaceEvolution}). The long-time behaviour of the dynamic solutions for Cases I and II  are indistinguishable within numerical resolution, as expected.

\begin{figure}
\centering
\includegraphics[width=4.3cm]{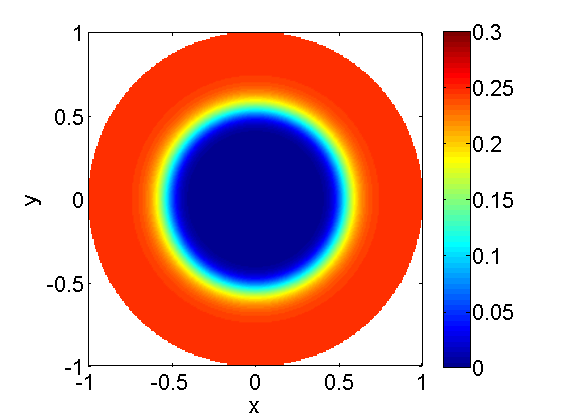}\includegraphics[width=4.3cm]{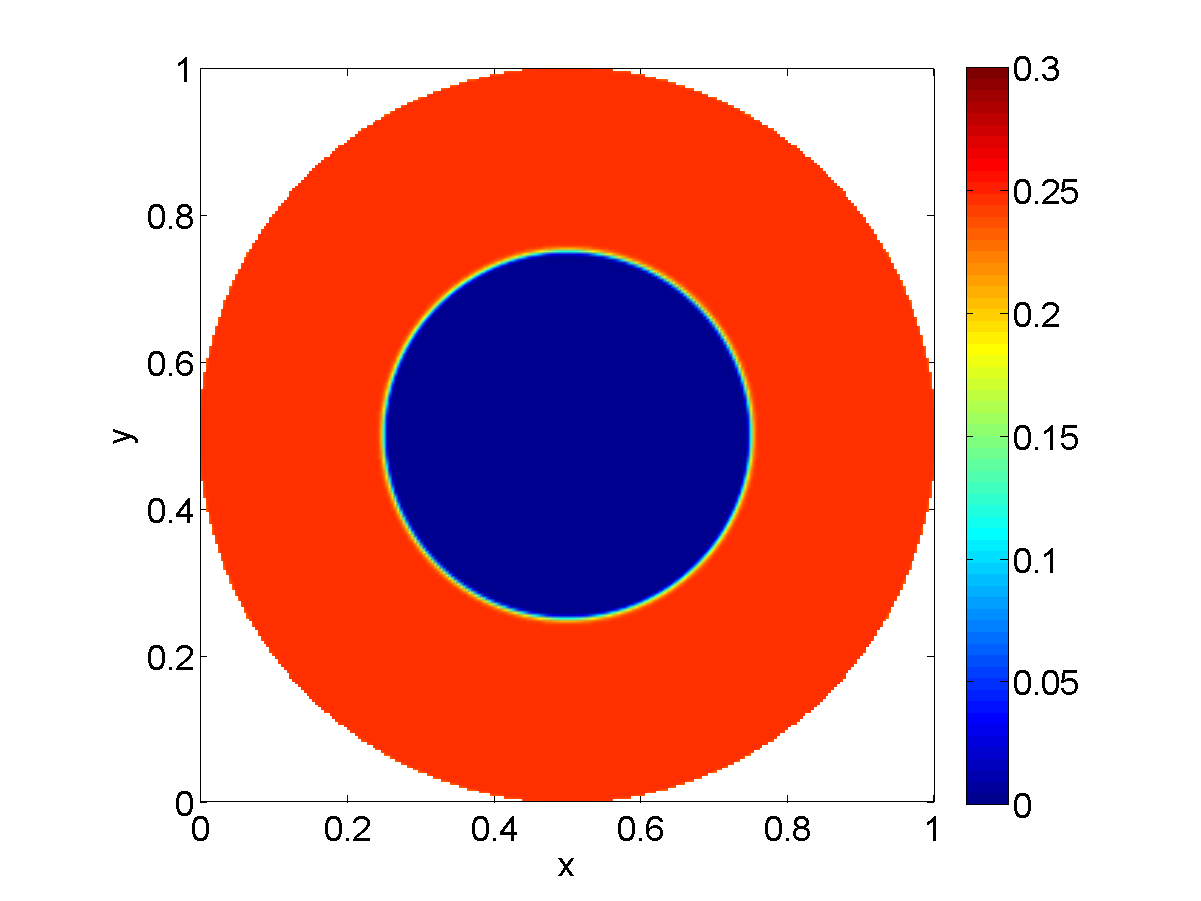}\includegraphics[width=4.3cm]{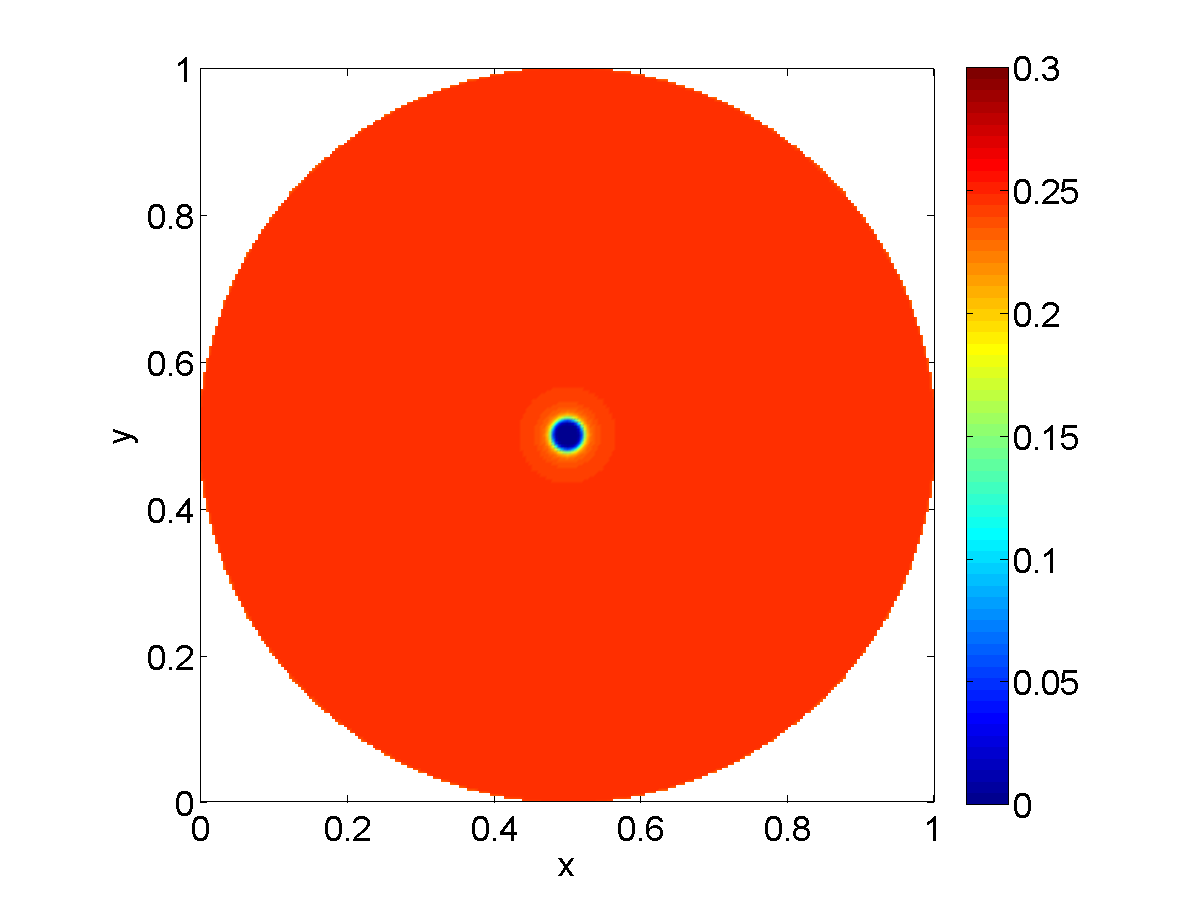}
\caption{$|\boldsymbol{Q}(\rvec,t)|^2$ on the cross section of the unit ball at $\phi=0$ for Case I and Case II, at $t=0$, $t=0.001$ and $t=0.125$. The spatial resolution is $h=\frac{1}{256}$.}
\label{Interface}
\end{figure}

\begin{figure}
\centering
\includegraphics[width=6cm]{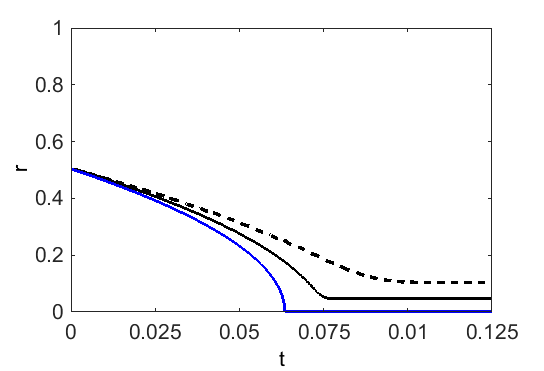}
\caption{Interface position for Case I and Case II with $r_0=0.5$ (black), and predicted position according to motion by mean curvature (blue) for $\tilde{L}=0.05$ (dashed) and $\tilde{L}=0.01$ (solid). In Case II an interface quickly develops so the two curves are indistinguishable.
The spatial resolutions for $\tilde{L}=0.05$ and $\tilde{L}=0.01$ are $h=\frac{1}{128}$ and $\frac{1}{256}$ respectively.}
\label{InterfaceEvolution}
\end{figure}

\begin{figure}
\centering
\includegraphics[width=4.3cm]{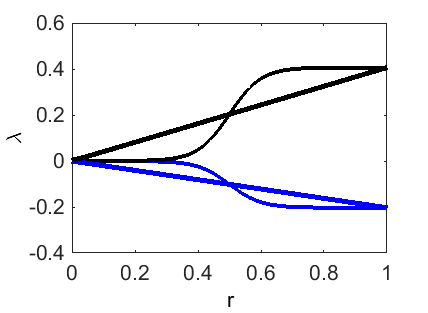}\includegraphics[width=4.3cm]{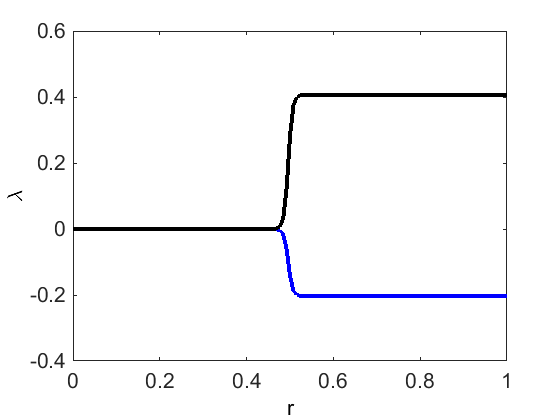}\includegraphics[width=4.3cm]{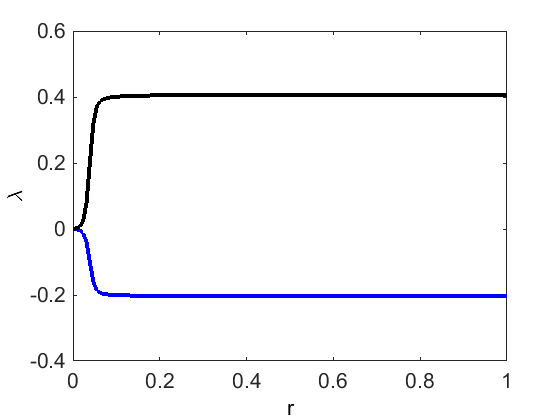}
\caption{Eigenvalues of $\boldsymbol{Q}(\mathbf{r},t)$, as a function of $r$, for Case I  and Case II (dotted), at $t=0$, $t=0.001$ and $t=0.125$.}
\label{Biaxiality}
\end{figure}

\begin{figure}
\centering
{\includegraphics[width=4.3cm]{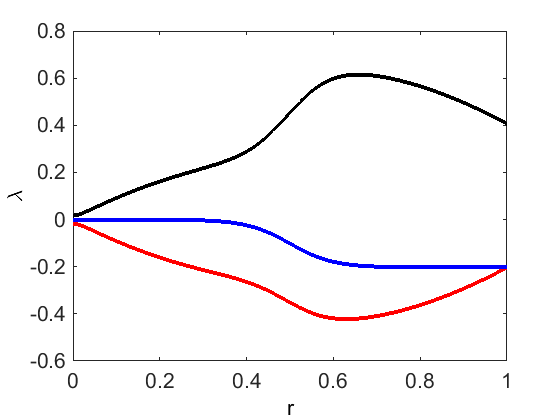}\includegraphics[width=4.3cm]{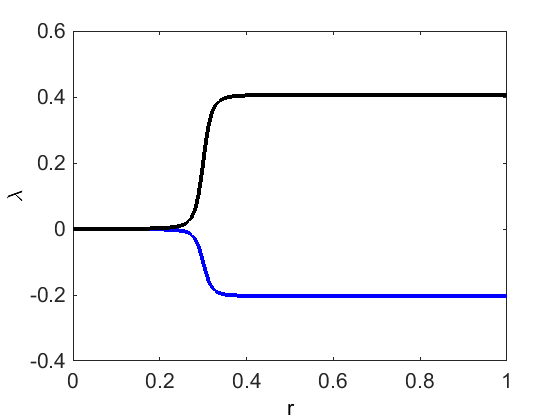}\includegraphics[width=4.3cm]{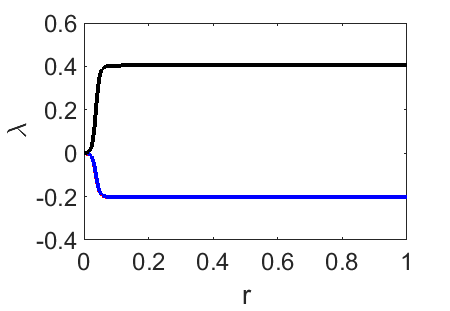}}
\caption{Eigenvalues of $\boldsymbol{Q}(\rvec,t)$, as a function of $r$ for initial condition \eqref{biaxsphereIC}, at $t=0$, $t=1.5 \times 10^{-4}$ and $t=0.085$.}
\label{BiaxialityBiax}
\end{figure}
$\mbox{}$\\
For \textbf{Case III}, we use a biaxial initial condition given by
\begin{equation}
\label{biaxsphereIC}
\boldsymbol{Q}(\rvec,0)=h(r)\left(\hat{\boldsymbol{r}}\otimes \hat{\boldsymbol{r}} -\frac{\boldsymbol{I}}{3}\right)+s(r)\left(\boldsymbol{m}\otimes \boldsymbol{m}-\boldsymbol{p} \otimes \boldsymbol{p} \right),
\end{equation}
where $\boldsymbol{m}=\left(\cos \theta \cos \phi,\cos \theta \sin \phi, -\sin \theta \right)$ and $\boldsymbol{p}=\left(-\sin \phi, \cos \phi, 0\right)$. 
The function $h(r)$ has an interface structure, as in Case I, and $s(r)=r(1-r)$. This case is outside the scope of Theorem~\ref{thm:1} and we are not guaranteed the uniaxial radial symmetry of the dynamic solution. The numerics show that the dynamic solution quickly becomes uniaxial within numerical resolution, as demonstrated by the evolution of the eigenvalues of $\Qvec(\rvec, t)$ in Figure \ref{BiaxialityBiax}. The dynamic solution exhibits an inwards-propagating interface separating the isotropic core at $r=0$ from the ordered nematic state and the interface equilibrates near the origin. We also numerically compute the differences
$$ \left| \frac{Q_{ij}(\rvec, t)}{|\Qvec|} -\left( \frac{\xvec_i \xvec_j}{|\rvec|^2} - \frac{\delta_{ij}}{3} \right)\right| $$ as a function of time and our numerics show that this difference vanishes everywhere away from the origin, since $|\Qvec(0, t)| = 0$ for large times; see Figure \ref{Q1ComparisonBiax}.\\

\begin{figure}
	\centering
	\includegraphics[width=4.3cm]{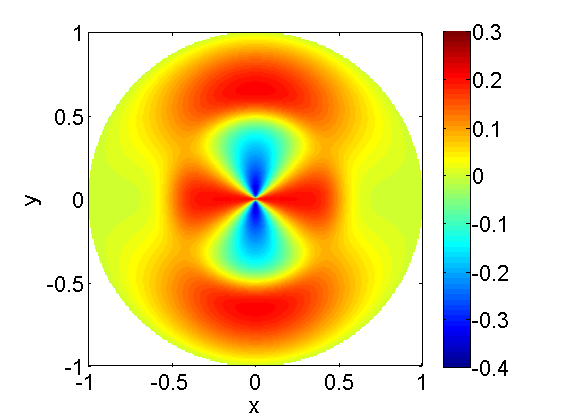}\includegraphics[width=4.3cm]{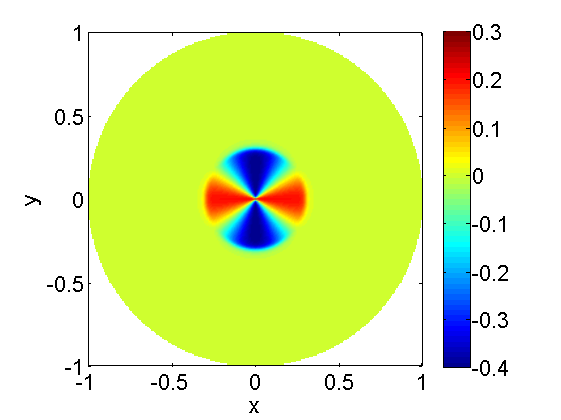}\includegraphics[width=4.3cm]{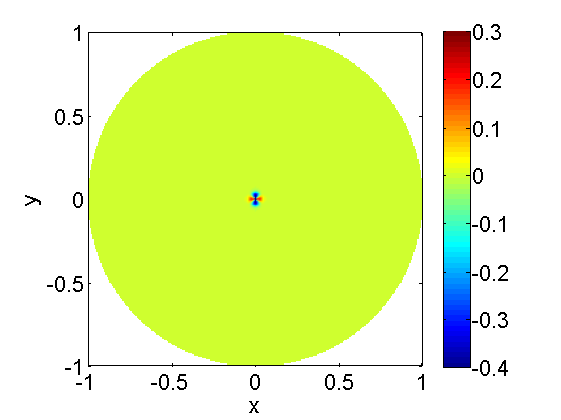}
	\caption{$Q_{11}(\rvec,t)-h_+((\cos \phi \sin \theta)^2-1/3)$ on the cross section of the unit ball for initial condition \eqref{biaxsphereIC}, at $t=0$, $t=1.5 \times 10^{-4}$ and $t=0.085$. The spatial resolution is $h=\frac{1}{256}$.}
	\label{Q1ComparisonBiax}
\end{figure}

%
  \textbf{Case IV} breaks the radial symmetry of the initial order parameter by employing an uniaxial initial condition of the form (\ref{eq:rh1}) with
  \small{
  \begin{equation}
  \label{eq:breaksymm}
  h(r, \theta, \phi, 0) = \frac{B}{6C}\left( 1 + \tanh\left( \frac{r^2\sin^2\theta( \cos^2\phi + 4  \sin^2\phi) + 2 r^2 \cos^2\theta - 0.5}{\sqrt{\tilde{L}}}\right) \right).
  \end{equation}   }
 \normalsize 
 Here, the initial interface is ellipsoidal in shape. The interface becomes circular and the subsequent dynamics are indistinguishable from Case I as seen in Figure \ref{ModQEllipse}.\\
  
\begin{figure}
	\centering
	\includegraphics[width=4.3cm]{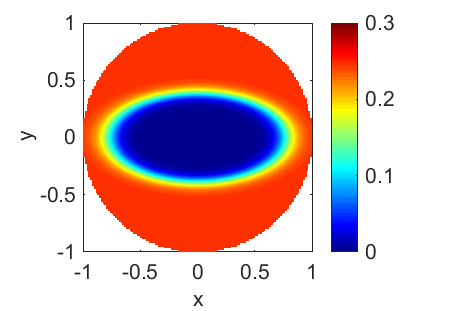}
	\includegraphics[width=4.3cm]{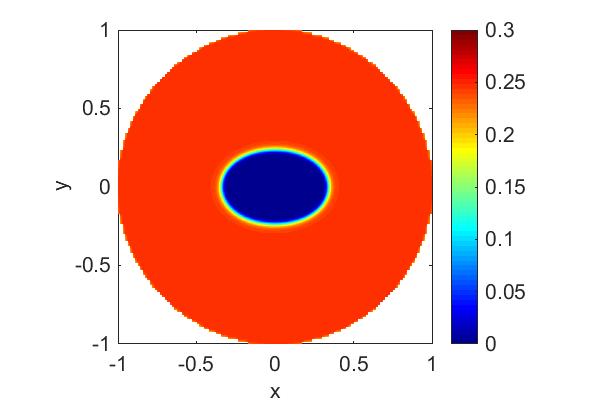}\includegraphics[width=4.3cm]{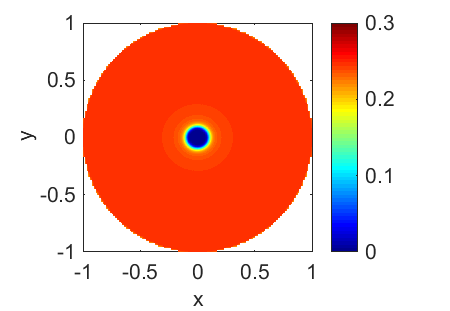}
	\caption{$|\boldsymbol{Q}(\rvec,t)|^2$ on the cross section of the unit ball for initial condition (\ref{eq:breaksymm}), at $t=0$, $t=0.05$ and $t=0.1$. The spatial resolution is $h=\frac{1}{256}$.}
	\label{ModQEllipse}
\end{figure}  
  
 These examples illustrate that whilst the static RH solution has a localized defect core of reduced order near the origin that may not be experimentally observable, the transient solutions exhibit well-defined isotropic-nematic interfaces (see Figure~\ref{InterfaceEvolution}). These interfaces propagate towards the droplet centre and may be experimentally observable.


\section{Fronts on a disc}
\label{sec:cylinder}
\subsection{Analysis on a disc}
We take our computational domain to be the unit disc defined by 
\begin{equation}
\label{eq:c1}
\Omega = \left\{ (r, \theta) \in \Rr^2; \,0\leq r
\leq 1,\, 0\leq \theta < 2\pi \right\},
\end{equation}
with fixed boundary condition $\Qvec = \Qvec_c = \frac{B}{3C}\left(\nvec_1 \otimes \nvec_1 - \frac{\mathbf{I}}{3} \right)$ on $r=1$
for $\nvec_1 = \left(\cos \theta, \sin\theta, 0 \right)$, the radial unit vector. This boundary condition is purely uniaxial and is a minimum of
the bulk potential. We study dynamic
solutions of the parabolic system, (\ref{eq:6}), subject to 
$\Qvec=\Qvec_c$ on $r=1$ with different types of initial conditions. We note that such solutions also survive as translationally invariant solutions, independent of $z$, on a cylinder with free boundary conditions on the top and bottom surfaces. 

We first present some heuristics based on critical points of the LdG energy on a disc
subject to this Dirichlet condition.
The gradient flow model dictates that dynamic
solutions evolve along a path of decreasing energy, converging to
a critical point of the LdG energy \cite{gradientflow}.
Hence, the long-time behaviour can be predicted by an analysis of
the corresponding stationary problem. In \cite{amaz}, the authors present a general analysis of
LdG energy minimizers on 3D nice domains(see (\ref{eq:2})), in the $L \to 0$ limit.
Based on their analysis, the minimizers converge strongly in $W^{1,2}\left(\Omega; S_0 \right)$ to a limiting harmonic map of the form
$ \Qvec = s\left(\nvec^* \otimes \nvec^* -\frac{\mathbf{I}}{3} \right)$ 
such  that $s=0$ or $s={B}/{3C}$ a.e. (so that $\Qvec$ is a minimum of
$f_B$) and the director $\nvec^*$ is a solution of the harmonic map equations
$ \Delta \nvec^* + |\nabla \nvec^*|^2 \nvec^* = 0$ (also see \cite{montero} for recent work on planar domains). The convergence is shown to be uniform away from the singularities of the limiting harmonic map, which need not be unique. There are at least two harmonic maps on a disc with the boundary condition 
$ \nvec^* = \left( \cos \theta, \sin \theta, 0 \right)$ on $r=1$ \cite{bethuel}, 
\begin{eqnarray}
\label{eq:c6} \nvec_1 = \left(\cos \theta, \sin\theta, 0
\right),\,\, \,\mathrm{and} \,\,\, \nvec_2 = \left(\frac{2x}{1 + r^2}, \frac{2y}{1 + r^2}, \frac{1
- r^2}{1 + r^2} \right).
\end{eqnarray}
We conjecture that there are at least two competing limiting harmonic maps for this two-dimensional problem, defined in terms of $\nvec_1$ and $\nvec_2$ above:
\begin{eqnarray}
\label{eq:c8}  \Qvec_1 = s\left(\nvec_1 \otimes \nvec_1 -
\frac{\mathbf{I}}{3} \right), \,\,\, \mathrm{and}\,\,\,\Qvec_2 = \frac{B}{3C}\left(\nvec_2 \otimes \nvec_2 -
\frac{\mathbf{I}}{3} \right).
\end{eqnarray}
As $\nvec_1$ is not defined at $r=0$, $\Qvec_1$ must have an isotropic point at $r=0$, with $s \to {B}/{3C}$
rapidly away from $r=0$. We refer to $\Qvec_1$ as being the
two-dimensional planar radial hedgehog profile. However, $\nvec_2$ has no singularity on $\Omega$ and $\Qvec_2$ does not have an isotropic core. We
predict that the dynamic solutions of (\ref{eq:6}), subject to
$\Qvec=\Qvec_c$ on $r=1$, converge to $\Qvec_1$ everywhere away from the origin (where $\nvec_1$ is singular) if escape into the $z$-direction is not allowed. Equivalently, $|\Qvec_{L}(\rvec, t) - \Qvec_1 (\rvec, t) | \to 0$ as $L \to 0$ for $|\rvec|>\sigma(L)$, $\sigma(L) \to 0$ where $L \to 0$ (or as $\frac{LC}{R^2 B^2} \to 0$) but $\Qvec_1$ is not a stationary solution of (\ref{eq:6}). The dynamic solutions converge to $\Qvec_2$ if escape into third dimension is allowed. Next, we have a lemma which demonstrates that escape into third dimension is not allowed for certain initial conditions.
We refer to a $\Qvec$-tensor as being ``planar" if the components $Q_{13},
Q_{23} = 0$ are identically zero on $\Omega$ and ``non-planar" if not. In
particular, $\Qvec_c$, is a planar
$\Qvec$-tensor.\\ 

\begin{lemma}
Let $\Qvec(\rvec, t)$ be a solution of (\ref{eq:6}) on $\Omega$,
with fixed boundary condition $\Qvec=\Qvec_c$ on $r=1$ and a planar initial condition,
$\Qvec(\rvec,0)$ such that $|\Qvec(\rvec, 0)| \leq
\sqrt{\frac{2}{3}}\left(\frac{B}{3C}\right)$ for $\rvec \in \Omega$. Then
$Q_{13} = Q_{23} = 0$ for all $t\geq 0$.
\end{lemma}

\begin{proof}The proof is an immediate application of Gronwall's
inequality \cite{gronwall}. From \cite{MajumdarLiquidCrystal, amaz}, we
have the following $L^{\infty}$-bound for the dynamic solution,
$\left| \Qvec \left(\rvec, t \right) \right|
 \leq \sqrt{\frac{2}{3}}\left(\frac{B}{3C}\right)$ for $t \geq 0$.
 The two governing PDEs for $Q_{13}$ and $Q_{23}$ can be written
 in the form
 \begin{eqnarray}
 \label{eq:c10}
 && \frac{\partial Q_{13}}{\partial t} - L \Delta Q_{13} =
 F(\Qvec)Q_{13} + B Q_{12}Q_{23} \nonumber \\
 && \frac{\partial Q_{23}}{\partial t} - L \Delta Q_{13} =
 G(\Qvec) Q_{23} + BQ_{12}Q_{13}
 \end{eqnarray}
 where $F$ and $G$ are bounded functions by virtue of the $L^{\infty}$ bound above. 
 We integrate by parts, use the fact that $Q_{13}=Q_{23} = 0$ on $r=1$ and 
 apply Gronwall's inequality to obtain
 \begin{eqnarray}
\label{eq:c13} \left( \int_{\Omega} Q_{13}^2 +
 Q_{23}^2~\mathrm{d}V\right) \leq \exp\left[\delta_4 t \right] \left(\int_{\Omega} Q_{13}^2 +
 Q_{23}^2~\mathrm{d}V\right)\bigg|_{t=0} = 0,
 \end{eqnarray}
 so that $Q_{13} = Q_{23} = 0$ for all $t\geq 0$. 
 \end{proof}

 \subsubsection{Radially symmetric static solutions}
 \label{sec:uv}
We consider a particular class of planar critical points of the LdG energy on a disc, introduced in \cite{fritta} for low temperatures (described by $A<0$ in (\ref{eq:3})), referred to as \textbf{radially symmetric solutions}. The theoretical results in this section are a generalization of the results in \cite{fritta} to the nematic-isotropic transition temperature. We work in a different temperature regime where the LdG bulk potential has two equal energy minima, and hence we cannot a priori assume that the results in \cite{fritta, fritta2} apply to our case. These solutions are labelled by two order parameters $u$ and $v$ that only depend on the radial distance from the origin. We can perform some explicit analysis for these critical points which are good examples of planar initial conditions. 
We first show that there exist radially symmetric planar
$(u,v)$-critical points of the LdG energy on a disc at
the nematic-isotropic transition temperature,
which are unstable in the sense that the second variation of the LdG energy is negative for admissible perturbations.
We use these critical points to construct planar initial conditions for the LdG gradient flow model. We then use numerical simulations to corroborate our heuristics that the corresponding dynamic solutions quickly develop a radially symmetric isotropic-nematic interface (even when the initial data is not radially symmetric) and converge to $\Qvec_1$ everywhere away from the origin. Next we
investigate the effect of small ``non-planar" perturbations of
the planar  initial conditions and show that the solution 
follows the planar dynamics for small but noticeable times, followed by an
abrupt escape into the third dimension at $r=0$ and long-time
convergence to $\Qvec_2$. As in \cite{fritta}, we study the LdG Euler-Lagrange equations,
 \begin{equation}
 \label{eq:c15}
 L\Delta \Qvec = \frac{B^2}{27 C}\Qvec - B \left(\Qvec \Qvec - \frac{|\Qvec|^2}{3}\mathbf{I} \right) + C |\Qvec|^2 \Qvec
 \end{equation}
 and look for solutions of the form
 \begin{equation}
 \label{eq:c14a}
 \Qvec = \frac{u(r)}{2}\left( \nvec_1 \otimes \nvec_1 - \mvec \otimes \mvec \right) + v(r)\left(\pvec\otimes \pvec - \frac{\mathbf{I}}{3} \right),
 \end{equation}
 where $\mvec = \left(-\sin\theta, \cos \theta, 0 \right)$ and $\pvec = (0,0,1)$.
 The critical points (\ref{eq:c14a}) are referred to as $k$-radially symmetric solutions in \cite{fritta,fritta2} with $k=2$; we note that the authors normalize the tensors in (\ref{eq:c14a}) whilst we choose not to normalize the tensors for ease of presentation.
It is straightforward to verify that solutions of the form (\ref{eq:c14a}) exist if the functions $u$ and $v$ satisfy the following system of coupled second-order ordinary differential equations
 \begin{eqnarray}
 \label{eq:c16}
&&u''(r) +\frac{u'(r)}{r} -\frac{4u(r)}{r^2} =\frac{u}{L}\left(\frac{B^2}{27C}  +\frac{2}{3}Bv+C\left(\frac{u^2}{2}+\frac{2v^2}{3}\right) \right),\\
\label{eq:c17}
&&v''(r) +\frac{v'(r)}{r}= \frac{v}{L}\left(\frac{B^2}{27C} -\frac{Bv}{3}+C\left(\frac{u^2}{2}+\frac{2v^2}{3}\right) \right)+\frac{1}{4L}Bu^2,
\end{eqnarray}
with $u(0)=v'(0)=0$ with $u(R)={B}/{3C}$ and $v(R)=-{B}/{6C}$, consistent with the boundary condition $\Qvec=\Qvec_c$ on $r=1$. As in \cite{fritta}, we can prove the existence of a solution pair $(u,v)$ of (\ref{eq:c16})-(\ref{eq:c17}) by appealing to a variational problem.
Define the energy
\begin{align}
\label{uvenergy}
\mathcal{E}(u,v)=\int^1_0& \bigg( \left( \frac{1}{4}(u')^2+\frac{1}{3}(v')^2+\frac{1}{r^2}u^2 \right) +\frac{B^2}{54CL}\left(\frac{u^2}{2}+\frac{2v^2}{3}\right) \nonumber\\
&+\frac{C}{L}\left(\frac{u^4}{16}+\frac{u^2 v^2}{6}+\frac{v^4}{9}\right)-\frac{B}{3L} v\left(\frac{2v^2}{9}-\frac{u^2}{2}\right)\bigg) \, r\mathrm{d}r.
\end{align}
This is the LdG energy of the $(u,v)$-ansatz in (\ref{eq:c14a}), defined on the admissible set
$S=\{ (u,v):[0,1] \to \mathbb{R}^2 | \sqrt{r} u', \, \sqrt{r}v',\,{u}/{\sqrt{r}},\,\sqrt{r}v \in L^2(0,1),\,u(1)={B}/{3C},\,v(1)=-{B}/{6C}\}$. The proof of the following lemma is standard and omitted here for brevity (see \cite{fritta}).
\\
\begin{lemma}
For each $L>0$,  there exists a global minimiser $(u,v) \in [C^{\infty}(0,1) \cap C([0,1])] \times [C^{\infty}(0,1)\cap C^1([0,1])]$ of the energy \eqref{uvenergy} on S, which satisfies the ODEs for $u$ and $v$ in \eqref{eq:c16}-\eqref{eq:c17}.
\end{lemma}

Next, we look at some qualitative properties of the $(u,v)$-solutions. Similar questions have been addressed in the recent paper \cite{fritta2}, in the temperature regime $A<0$, with the exception of the monotonicity argument in Lemma~\ref{lem:mon} below and we reproduce all necessary details for completeness. 
\\
\begin{lemma}
Let $(u,v)$ be a global minimizer of the energy $\mathcal{E}$ in (\ref{uvenergy}) subject to $u(1)=\frac{B}{3C}$ and $v(1) = -\frac{B}{6C}$. Then $0\leq u\leq \frac{B}{3C}$ and $-\frac{B}{6C}\leq v\leq 0$ for $0\leq r \leq 1$.
\end{lemma}

\begin{proof}
We can prove the non-negativity of $u$ by following the arguments in \cite{fritta, fritta2} (the authors use the symmetry $E\left[u, v \right] = E\left[ -u, v \right]$ and the strong maximum principle for (\ref{eq:c16}) to deduce that $u\geq 0$ since $u(1)>0$). We assume $v(r_1) = v(r_2) =0$ and $v>0$ for $r_1 < r < r_2$. Define the perturbation
\begin{eqnarray}
\label{eq:c18}
\bar{v}(r) =
\begin{cases}
v(r) & 0\leq r \leq r_1 \\
0 & r_1 < r < r_2 \\
v(r) & r_2 \leq r \leq 1.
\end{cases}
\end{eqnarray}
A direct computation shows that
\begin{equation}
\label{eq:c19}
E[u,v] - E[u, \bar{v}] = \int_{r_1}^{r_2} r \frac{v_r^2}{3} + \frac{r}{L}\frac{C v^2}{9}\left(v - \frac{B}{3C}\right)^2 + r \frac{u^2}{6 L}\left( B v + C v^2 \right)~ \mathrm{d}r>0
\end{equation}
contradicting the global minimality of the pair $(u,v)$.
Next, let us assume for a contradiction that the minimum value of $v:[0,1] \to \Rr$, say $v_{\min} < -\frac{B}{6C}$ at  $r=r_0$. At $r=r_0$, the left-hand side of (\ref{eq:c17}) is non-negative by definition of a minimum. From the maximum principle (see \cite{ejam2010}, \cite{amaz}), we have
\begin{equation}
\label{eq:c21}
|\Qvec|^2 = \frac{u^2}{2} + \frac{2 v^2}{3} \leq \frac{2}{3}\left(\frac{B^2}{9C^2}\right).
\end{equation}
If $v_{\min} < -\frac{B}{6C}$, then $u^2(r_0) < \frac{B^2}{9C^2}$.
Then for $v_{min} < -\frac{B}{6C}$, we have
\begin{eqnarray}
\label{eq:c20}
&& \frac{B^2}{27 C} v_{min} - \frac{B}{3}v_{min}^2 + \frac{2C}{3} v_{min}^3 < -\frac{B^3}{54 C^2} \quad \textrm{and} \nonumber \\ && \frac{u^2}{4}\left( B + 2C v_{min} \right) < \frac{B^3}{54 C^2}
\end{eqnarray} so that the right-hand side of (\ref{eq:c17}) is negative, yielding a contradiction. The result for $u$ follows in the same manner, using the bound for $v$ proven above.
\end{proof}

We next show that $(u,v)$ are monotone functions, borrowing an idea from \cite{lamyradial}. We make the elementary observation that $u' > 0$ for $0< r < \sigma$, since $u$ attains its minimum value at $r=0$.

\begin{eqnarray}
\label{eq:c22}
u_\eps(r) = u(r)+\eps \alpha(r), \quad v_\eps(r) = v(r) + \eps \beta(r)
\end{eqnarray}
with $\alpha\left(1 \right) = \beta\left(1 \right) = 0$.
In this case the second variation is given by
\begin{eqnarray}
\label{eq:c23}
\delta^2 E\left[\alpha, \beta \right]: = &&\int_{0}^{1}r \left[ \frac{\alpha_r^2}{4} + \frac{\beta_r^2}{3} + \frac{\alpha^2}{r^2} \right]+ \frac{r}{L}\left[\frac{B^2}{27 C}\frac{\beta^2}{3} - \frac{2Bv\beta^2}{9} + \frac{2 C}{3}v^2 \beta^2 \right]~\mathrm{d}r \nonumber \\
&& + \int_{0}^{1} \frac{r}{L} \left[ \frac{C}{6}u^2 \beta^2 + \frac{B}{6}v\alpha^2 + \frac{C}{6}v^2 \alpha^2 + \frac{B u \alpha \beta}{3} + \frac{ 2C u v \alpha \beta}{3} \right]~\mathrm{d}r+ \nonumber \\ && +\int_{0}^{1}\frac{r}{L}\left[ \frac{B^2}{108 C}\alpha^2 + \frac{3 C}{8}u^2 \alpha^2 \right]~\mathrm{d}r.
\end{eqnarray}
In particular, $\delta^2 E\left[\alpha, \beta \right] \geq 0$ for all admissible $\alpha, \beta$ by the global minimality of $(u,v)$.
\\
\begin{lemma}
\label{lem:mon}
Let $(u,v)$ be a global minimizer of the energy in (\ref{uvenergy}). Then $u' >0$ for $r>0$ and $v' < 0$ for $r>0$.
\end{lemma}

\begin{proof}
We assume for a contradiction that $u$ and $v$ are not monotone, so that there exist points $r_1, r_2, r_3, r_4 $ with $r_1, r_2 \in \left(0, 1 \right)$, $r_3\in \left[0,1 \right)$ such that
\begin{eqnarray}
\label{eq:c24}
&& u'(r_1) = u'(r_2) = 0 \quad \textrm{$u'<0$ for $r_1 < r < r_2$} \nonumber \\
&& v'(r_3) = v'(r_4) = 0 \quad \textrm{$v' >0$ for $r_3 < r < r_4$.}
\end{eqnarray}
We differentiate the ODEs for $u$ and $v$ in (\ref{eq:c16})-(\ref{eq:c17}), multiply by $r u'$ and $r v'$ respectively, integrate over $r\in[r_1, r_2]$ and $r \in [r_3, r_4]$ to get the following equalities:
\begin{eqnarray}
\label{eq:c25}
 \int_{r_1}^{r_2} \frac{r}{4}(u^{''})^2 + &&\frac{5}{4r}(u')^2 - \frac{2}{r^2} u u' + \frac{r}{L}\bigg[\frac{B^2}{108 C}(u')^2 + \frac{3C}{8}u^2 (u')^2 \nonumber\\
 &&+ \frac{C}{6}(u')^2 v^2 + \frac{B}{6}v (u')^2 + \frac{u u' v'}{6}(B + 2 C v) \bigg]~\mathrm{d}r = 0, 
\end{eqnarray}
and
\begin{eqnarray}
\label{eq:c26}
 \int_{r_3}^{r_4}\frac{r}{3}( v^{''})^2 + &&\frac{(v')^2}{3r} + \frac{r}{L}\left[\frac{B^2}{81 C}(v')^2- \frac{2B}{9}v (v')^2 + \frac{2C}{3}v^2(v')^2 \right]\nonumber \\ && + \frac{r}{L}\left[ \frac{C}{6}u^2 (v')^2 + \frac{u u' v'}{6}\left( B + 2 C v \right) \right]~\mathrm{d}r = 0.
\end{eqnarray}
We define the perturbations $\alpha, \beta$ as follows:
\begin{eqnarray}
\alpha(r) = \begin{cases} 0 & \textrm{if $u'\geq 0$} \\ u' & \textrm{if $u'< 0$},
\end{cases} \qquad
\beta(r) = \begin{cases} 0 & \textrm{if $v'\leq 0$} \\ v' & \textrm{if $v'> 0$} .
\end{cases}
\end{eqnarray}
These perturbations satisfy $\alpha(1) = \beta(1) = 0$, since $u$ and $v$ attain their maximum and minimum values respectively, on $r=1$.
Substituting $(\alpha, \beta)$ in (\ref{eq:c23}) and using (\ref{eq:c25})-(\ref{eq:c26}), we obtain $\delta^2 E\left[\alpha, \beta \right] <0$ and the required contradiction.
\end{proof}

Finally, we demonstrate that this class of critical points is unstable in the static sense, at the nematic-isotropic transition temperature.
We consider a perturbation about the critical point in (\ref{eq:c14a}), $\mathbf{W}_{ij}=\Qvec_{ij}+\epsilon \mathbf{V}_{ij}$ with
$\mathbf{V}_{ij}=0$ on $r=1$ and compute the second variation of the LdG energy about this critical point as shown below
\begin{equation}\label{eq:c27}
\delta^2 I=\int \int \int \frac{1}{2}|\nabla
\mathbf{V}|^2+\frac{A}{2L}|\mathbf{V}|^2
-\frac{B}{L}\Qvec_{ij}\mathbf{V}_{jp}\mathbf{V}_{pi}+\frac{C}{L}(\Qvec\cdot
\mathbf{V})^2+\frac{C}{2} |\Qvec|^2|\mathbf{V}|^2\, \mathrm{d}V.
\end{equation}
We simply use the perturbation
\begin{equation}
\label{eq:per}
\mathbf{V} = \frac{100r^2 (1 - r^2)^2}{(1 +100 r^2)}(\nvec_1\otimes \pvec
 +\pvec \otimes \nvec_1)
 \end{equation}
 in (\ref{eq:c27}). 
This integral is evaluated numerically using numerical solutions to (\ref{eq:c16}) - (\ref{eq:c17}) and  $\delta^2 I < 0$ for $\log_{10} L < -1.6$, as illustrated in the graph below.
\begin{figure}[h!]
\centering
\centerline{\includegraphics[width=5cm]{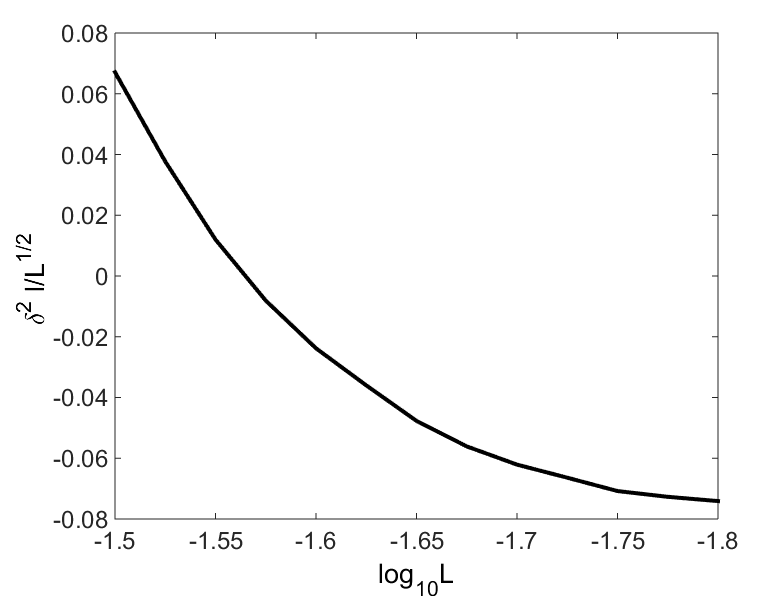}}
\caption{$\delta^2 I$ for various L.}
\label{secondvariation}
\end{figure}

\subsection{Numerical simulations}
\label{sec:dynamic}
We numerically solve the LdG gradient flow system in (\ref{eq:n1}-\ref{eq:n5}) on a disc subject to $\Qvec=\Qvec_c$ on $r=1$.
We take $R$ to be the disc radius in the definition of the dimensionless variables and the parameter values ($R$, $L$, $B$, $C$) are as in Section~\ref{sec:RH}. We take $\tilde{L}=0.01$ in this section, unless stated otherwise. The focus is on the distinction between planar and non-planar initial conditions. Based on the heuristics presented in Section~\ref{sec:uv}, we argue that all dynamic solutions  develop an interface separating an isotropic core at $r=0$, from an ordered uniaxial nematic phase, away from $r=0$. For planar initial conditions, the interface persists for all times and the dynamic solution has an isotropic core at $r=0$ whereas for non-planar initial conditions, the interface collapses at $r=0$ and the dynamic solution relaxes into the uniaxial state $\Qvec_2$ in (\ref{eq:c8}) for long times. 

Firstly, we consider planar $(u,v)$-type initial conditions of the form
\begin{equation}
\label{eq:c14}
 \Qvec(\rvec, 0) = \frac{u(r,0)}{2}\left( \nvec_1 \otimes \nvec_1 - \mvec \otimes \mvec \right) + v(r,0)\left(\pvec\otimes \pvec - \frac{\mathbf{I}}{3} \right) 
 \end{equation}
 where $\mvec = \left(-\sin\theta, \cos \theta, 0 \right)$, $\pvec = (0,0,1)$ and $\mathbf{I} = \nvec_1 \otimes \nvec_1 + \mvec\otimes \mvec + \pvec\otimes \pvec$ . Recall our boundary conditions enforce $u={B}/{3C}$ and $v=-{B}/{6C}$ on $r=1$. We let $u(r,0)$ and $v(r,0)$ have an interface structure as shown below:
\begin{eqnarray*}
u(r,0)=\frac{B}{6C}\left(1+\tanh\left(\frac{r-u_0}{\sqrt{\tilde{L}}}\right)\right),    \label{eq:c14bb1}
\,\, v(r,0)=-\frac{B}{12C}\left(1+\tanh\left(\frac{r-v_0}{\sqrt{\tilde{L}}}\right)\right), \label{eq:c14bb2}
\end{eqnarray*}
for various values of $u_0$ and $v_0$. It is worth noting that the corresponding dynamic solution is given by
\begin{equation}
\label{eq:c14aa}
\Qvec(\rvec, 0) = \frac{u(r,t)}{2}\left( \nvec_1 \otimes \nvec_1 - \mvec \otimes \mvec \right) + v(r,t)\left(\pvec\otimes \pvec - \frac{\mathbf{I}}{3} \right) 
\end{equation}
if the dynamic order parameters satisfy
\begin{eqnarray}
\label{eq:c14b}
&&\gamma u_t = u''(r) +\frac{u'(r)}{r} -\frac{4u(r)}{r^2} - \frac{u}{\tilde{L}}\left(\frac{B^2}{27C}  +\frac{2}{3}Bv+C\left(\frac{u^2}{2}+\frac{2v^2}{3}\right) \right),\\
&&\gamma v_t = v''(r) +\frac{v'(r)}{r} -  \frac{v}{\tilde{L}}\left(\frac{B^2}{27C} -\frac{Bv}{3}+C\left(\frac{u^2}{2}+\frac{2v^2}{3}\right) \right) - \frac{1}{4\tilde{L}}Bu^2.\label{eq:c14d}
\end{eqnarray}
 From Proposition~\ref{prop:1}, this is the unique solution for this model problem. In fact, we can go further and exploit the methods in \cite{bronsardstoth2} to compare the isotropic-nematic interface motion in (\ref{eq:c14aa})-(\ref{eq:c14d}) with mean curvature motion. As in Section~\ref{sec:RH}, we cannot quote results from \cite{bronsardstoth2} since the dynamic equations (\ref{eq:c14aa})-(\ref{eq:c14d}) differ from the Ginzburg-Landau model in \cite{bronsardstoth2} by the additional term $-{4u}/{r^2}$ in (\ref{eq:c14b}) above. However, for $\tilde{L}$ sufficiently small, this term  may be controllable and in Figure~\ref{uvinterface}, we plot the numerically computed interface location (plot $r^*(t)$ such that $|\Qvec(\rvec, t)|^2 < \frac{1}{3}h_+^2$ for $r<r^*(t)$) and find good agreement with mean curvature propagation for small times.


A typical solution with an initial condition of the form (\ref{eq:c14}) is shown in Figure \ref{Interfaceuvplanar}. 
If $u_0 \neq v_0$, then $\Qvec(\rvec, 0)$ is necessarily biaxial but it is hard to see the biaxial character of the initial data by looking at 
$|\boldsymbol{Q}|^2$. In order to see the rapid relaxation to uniaxiality we plot the eigenvalues of the dynamic solution, $\Qvec(\rvec, t)$, as a function of time (see Figure \ref{Biaxialityuvplanar}).
Varying the values of $u_0$ and $v_0$ does not change the qualitative dynamics: $\Qvec(\rvec, t)$ quickly becomes uniaxial for all choices of $u_0$ and $v_0$, within numerical resolution. The dynamic solution develops a radially symmetric interface separating the isotropic core, $\Qvec=0$ at $r=0$, from an ordered uniaxial nematic state (away from $r=0$) and the interface equilibrates near $r=0$ for long times. 
We have numerically computed the tensor-difference, $\Qvec(\rvec, t) - \Qvec_1$, as a function of time (where $\Qvec_1$ is defined in (\ref{eq:c8})) and find that $\Qvec(\rvec, t) \rightarrow \Qvec_1$, everywhere away from $r=0$, as expected. In particular, $\Qvec_1$ is not a solution of the LdG Euler-Lagrange equations and only describes the leading order behaviour away from $r=0$. One component of this tensor difference is plotted in 
Figure \ref{Q1Comparisonuvplanar}.

\begin{figure}
\centering
\includegraphics[width=4.3cm]{InterfaceStructureInitialSI.png}\includegraphics[width=4.3cm]{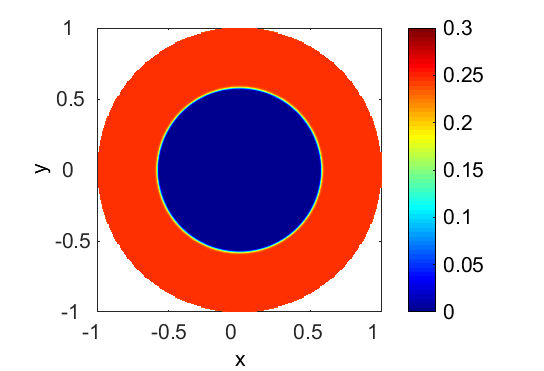}\includegraphics[width=4.3cm]{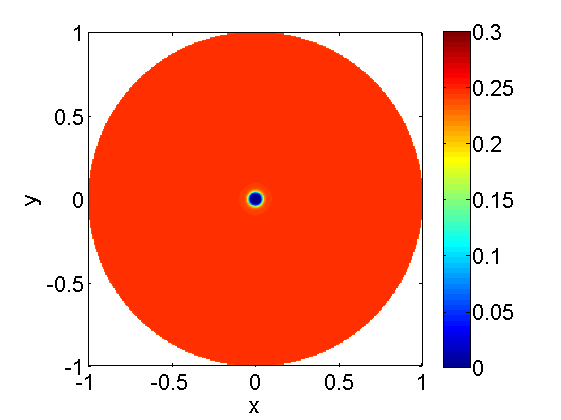}
\caption{$|\boldsymbol{Q}(\rvec,t)|^2$ on the cross section of the cylinder at $z=0$ for $u_0=0.6$, $v_0=0.4$ for the
initial condition \eqref{eq:c14}, at  $t=0$, $t=10^{-5}$, 
and $t=0.25$. The spatial resolution is $h=\frac{1}{256}$}
\label{Interfaceuvplanar}
\end{figure}


\begin{figure}
\centering
\includegraphics[width=4.3cm]{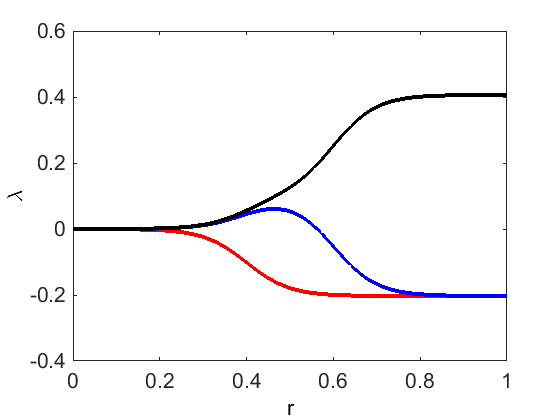}
	 \includegraphics[width=4.3cm]{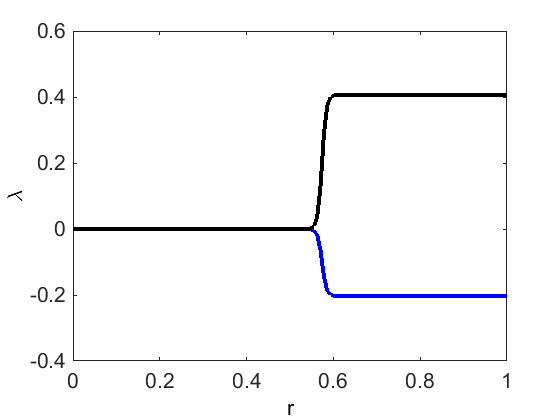}\includegraphics[width=4.3cm]{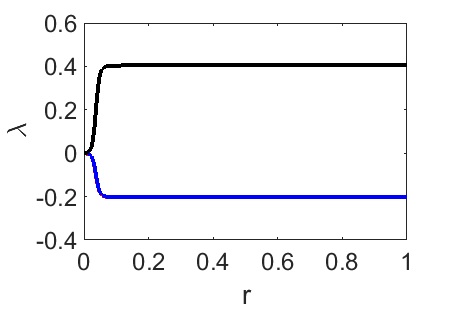}
\caption{Eigenvalues of $\boldsymbol{Q}(\mathbf{r},t)$ 
for initial condition \eqref{eq:c14} with $u_0=0.6$, $v_0 = 0.4$, at $t=0$, $t=10^{-5}$, 
 and $t=0.25$.}
\label{Biaxialityuvplanar}
\end{figure}


\begin{figure}
\centering
\includegraphics[width=4.3cm]{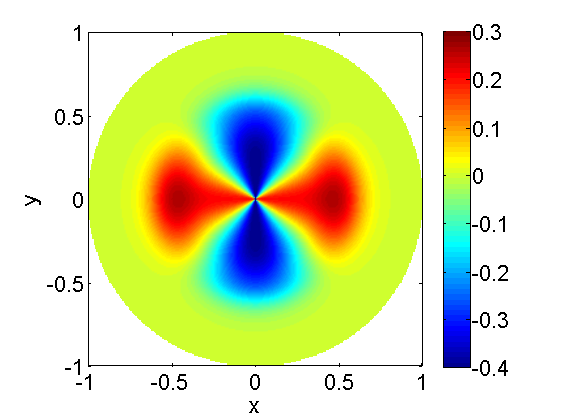}\includegraphics[width=4.3cm]{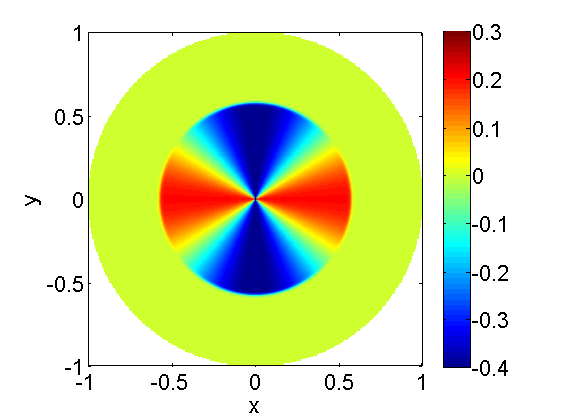}\includegraphics[width=4.3cm]{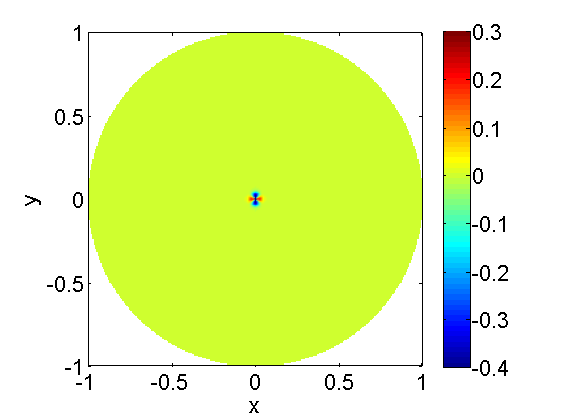}
\caption{$Q_{11}(\rvec,t)-\frac{B}{3C}(\frac{x^2}{r^2}-1/3)$ on the cross section of the cylinder at $z=0$ for the initial condition \eqref{eq:c14} with $u_0= 0.6$, $v_0 = 0.4$, at $t=0$, $t=10^{-5}$ and $t=0.25$.} \label{Q1Comparisonuvplanar}
\end{figure}

Next, we consider an initial condition of the form (\ref{eq:c14}) with
\begin{eqnarray}
\label{eq:c14c}
&& u(r,\theta,0) = \frac{B}{6C}\left( 1 + \tanh\left(\frac{r - 0.6\left(1 + 0.25 \sin 5 \theta \right)}{\sqrt{\tilde{L}}} \right) \right) \nonumber \\
&& v(r, \theta, 0) = - \frac{B}{12 C} \left( 1 + \tanh\left(\frac{r - 0.4\left(1 + 0.25 \sin 5 \theta \right)}{\sqrt{\tilde{L}}} \right) \right).
\end{eqnarray}
This is again a planar initial condition with an interesting star-shaped isotropic-nematic interface that relaxes into a radially symmetric isotropic-nematic interface. The subsequent dynamics is then indistinguishable from the case discussed above; this is illustrated by Figures~\ref{nonradialuv} and \ref{nonradialuvQ11}.

\begin{figure}
	\centering
	\includegraphics[width=4.3cm]{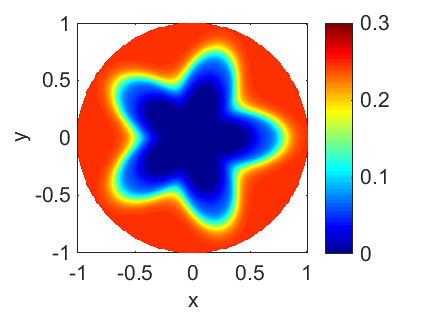}
	\includegraphics[width=4.3cm]{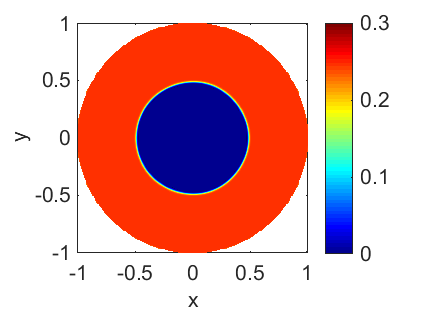}\includegraphics[width=4.3cm]{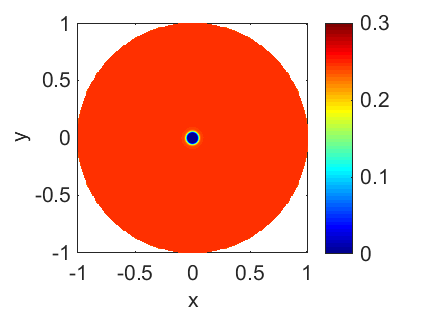}
	\caption{$|\boldsymbol{Q}(\rvec,t)|^2$ on $z=0$ for initial condition \eqref{eq:c14c} at $t=0$, $t=0.06$ and $t=0.2$. The spatial resolution is $h=\frac{1}{256}$.}
	\label{nonradialuv}
\end{figure}

\begin{figure}
	\centering
	\includegraphics[width=4.3cm]{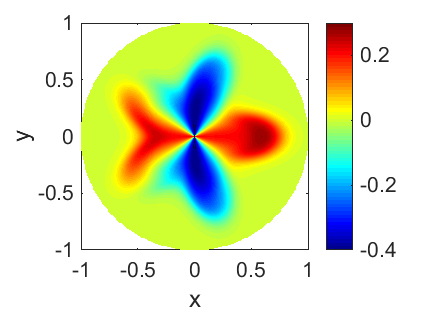}
	\includegraphics[width=4.3cm]{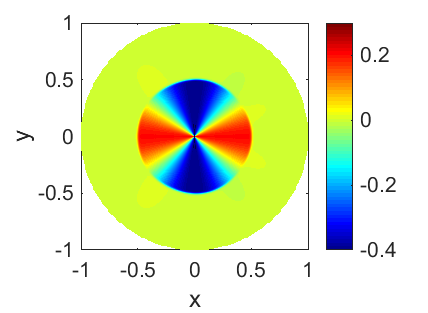}\includegraphics[width=4.3cm]{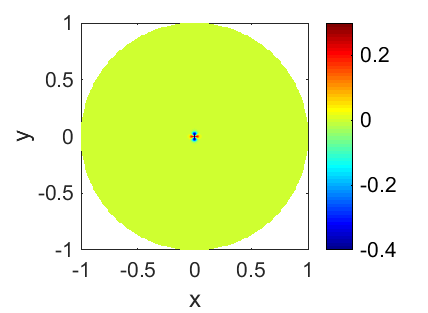}
	\caption{$Q_{11}(\rvec,t)-\frac{B}{3C}(\frac{x^2}{r^2}-1/3)$ on the cross section of the cylinder at $z=0$ for the initial condition \eqref{eq:c14c} at $t=0$, $t=0.06$ and $t=0.2$.}
	\label{nonradialuvQ11}
\end{figure}

Finally, we consider a small perturbation to the planar $(u,v)$-initial conditions, that render non-planar initial conditions. Let
\begin{equation}
\label{perturbedIC}
\bold{Q}(r,0)=u(r,0)\left(\bold{n} \otimes \bold{n}-\frac{\boldsymbol{I}_2}{2}\right)+v(r,0)\left(\bold{p}\otimes \bold{p}-\frac{\boldsymbol{I}}{3}\right),
\end{equation}
where,
\begin{align}
&\boldsymbol{n}=(\sqrt{(1-\epsilon^2(1-r)^2)}\cos\theta,\sqrt{(1-\epsilon^2(1-r)^2)}\sin\theta, \epsilon(1-r)),\nonumber\\
&\mathbf{I}_2 = \nvec_1 \otimes \nvec_1 + \mvec\otimes \mvec,\nonumber
\end{align}
and $\boldsymbol{p}$ is as before. The functions $u(r,0)$ and
$v(r,0)$ are as defined previously. As before, the dynamic
solution quickly becomes uniaxial (within numerical resolution)
irrespective of $u_0$ and $v_0$ and develops a well-defined interface
separating an interior region with $\Qvec=0$ near $r=0$, from an ordered uniaxial nematic
state elsewhere. This interface propagates inward but instead of being arrested at a small distance from the origin,  the interface collapses at the origin and the dynamic solution relaxes to $\Qvec_2$ in (\ref{eq:c8}). In particular, $|\Qvec(\rvec, t)|^2 \to \frac{2}{3}h_+^2$ uniformly on the cylinder for large times. 
Figures \ref{Biaxiality4} and \ref{Biaxiality409} show the time snapshots of the spatial distribution of eigenvalues, and the time 
evolution of the eigenvalues at $r=0$, showing the convergence to a uniform uniaxial solution throughout the cross-section.  
Figure \ref{Q1ComparisonP1} shows snapshots of the relaxation of one component of $\Qvec$ to the corresponding $\Qvec_2$-component. 

We also study how the initial non-planarity (as measured by $\epsilon$) affects the characteristic relaxation time to $\Qvec_2$.
We observe that the modulus, $|\Qvec|^2\left(0, t \right)$ jumps abruptly from zero to
$\frac{2}{3}h_+^2$ at some critical time. Let $t^*$ be the first
time for which
$$ \left|\Qvec\left(r, t^* \right)\right|^2 > \frac{1}{3} h_+^2, $$
and we associate $t^*$ with the loss of interface structure. Figure \ref{Interfacelost} plots $t^*$ as a function of 
$-\log_{10} \epsilon$ for various $u_0$ and $v_0$, and we find that $t^*\propto -\log_{10}\epsilon$. This can give quantitative estimates for the real-time persistence of isotropic-nematic interfaces and their experimental relevance for model problems with non-planar initial conditions.

\begin{figure}
\centering
\centerline{\includegraphics[width=4.7cm]{EigenvaluesInitialCU06andV04.png}\includegraphics[width=4.7cm]{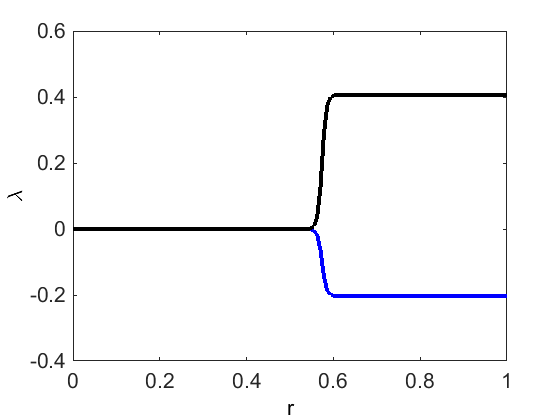} \includegraphics[width=4.7cm]{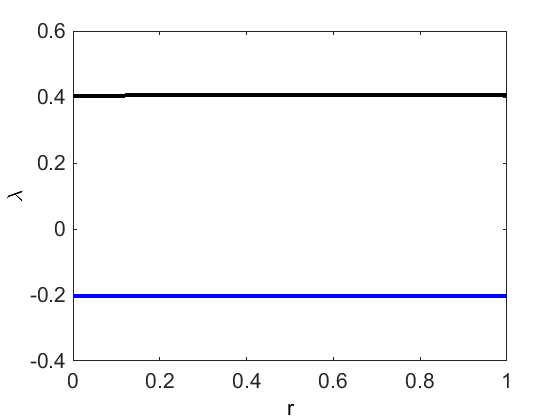}}
\caption{Spatial dependence of eigenvalues of $\boldsymbol{Q}(\mathbf{r},t)$ at  $t=0$, $t=0.001$ and $t=0.25$ for 
 initial condition \eqref{perturbedIC} with $u_0=0.6$, $v_0=0.4$.} 
\label{Biaxiality4}
\end{figure}

\begin{figure}
\centering
\includegraphics[width=3cm]{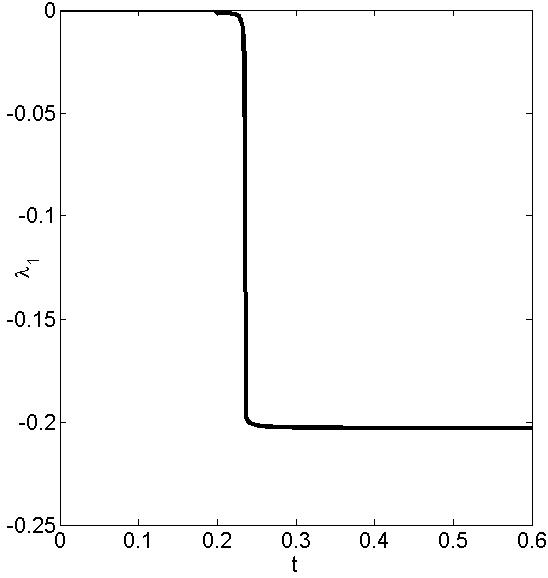}\quad \includegraphics[width=3cm]{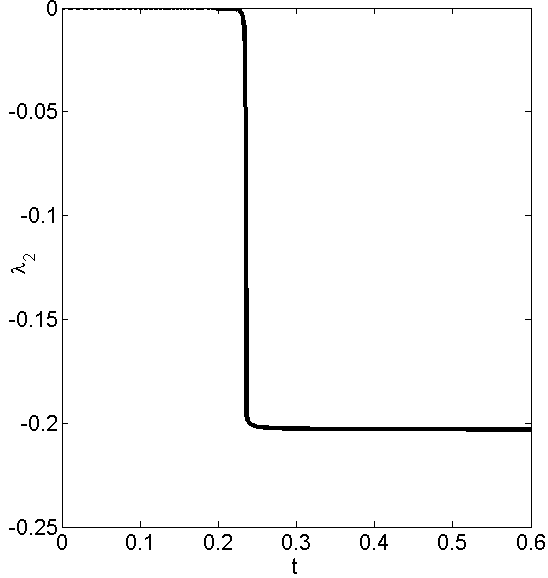}\quad \includegraphics[width=3cm]{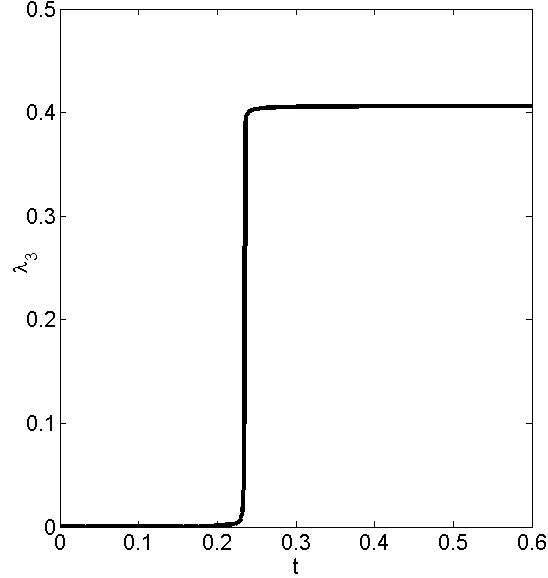}
\caption{Time evolution of eigenvalues of $\boldsymbol{Q}(\rvec,t)$ at the origin for initial condition  \eqref{perturbedIC} with 
$u_0=0.6$, $v_0=0.4$. }
\label{Biaxiality409}
\end{figure}

\begin{figure}
\centering
\includegraphics[width=4.3cm]{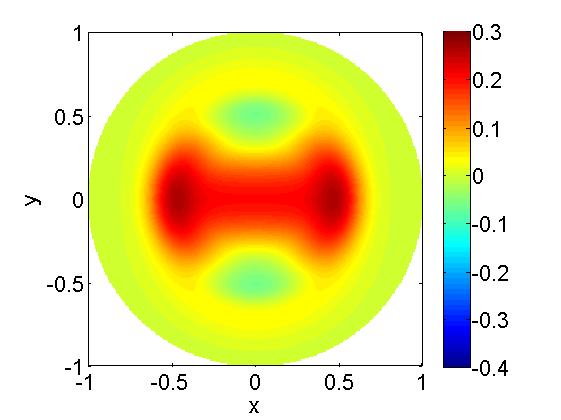}\includegraphics[width=4.3cm]{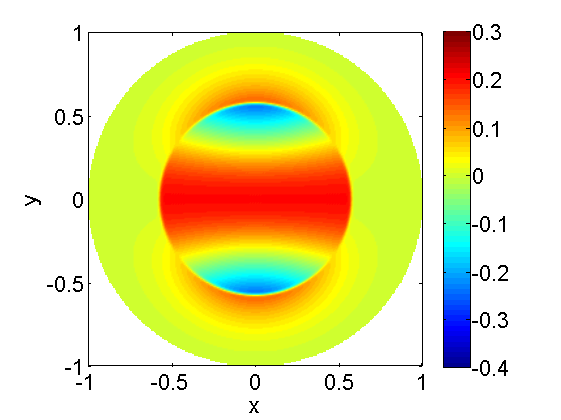}\includegraphics[width=4.3cm]{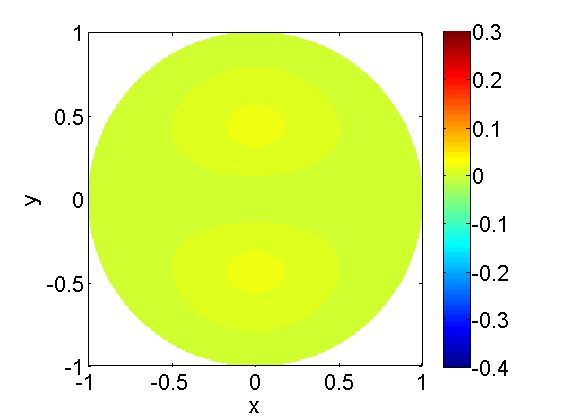}
\caption{$Q_{11}(\rvec,t)-\frac{B}{3C}(4x^2/(1+r^2)^2-1/3)$ on the cross section of the cylinder at $z=0$  for initial condition \eqref{perturbedIC} with $u_0 = 0.6$, $v_0 = 0.4$ at
$t=0$, $t=0.001$ and $t=0.6$. } \label{Q1ComparisonP1}
\end{figure}

\begin{figure}
	\centering
	\includegraphics[width=6cm]{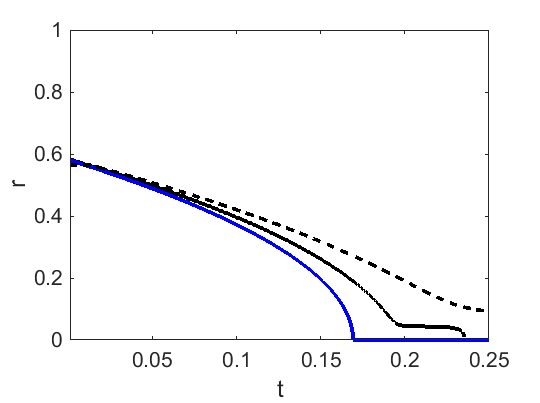}
	\caption{Interface position of dynamic solution (\ref{eq:c14bb2}) with initial condition (\ref{eq:c14}), with $u_0 = 0.6$, $v_0 = 0.4$ for $\tilde{L}=0.05$ (grey) and $\tilde{L}=0.01$ (black) compared to motion by mean curvature (blue).} 
	\label{uvinterface}
\end{figure}

\begin{figure}
\centering
\includegraphics[width=7cm]{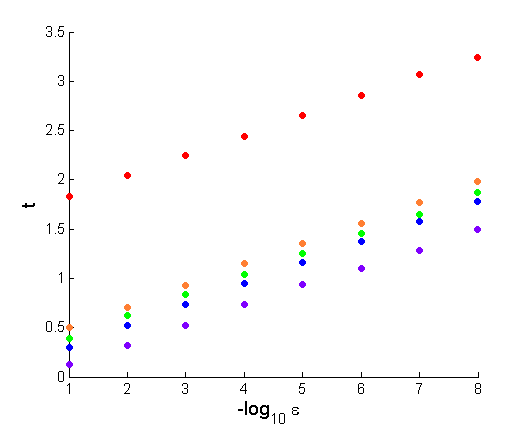}
\caption{Time ($t^*$) at which the interface is lost for various $\epsilon$ and $u_0$ and $v_0$; $u_0=0.1$ (purple), $u_0=0.4$ (blue), $u_0=0.5$ (green), $u_0=0.6$ (orange) and $u_0=0.9$ (red), $v_0=1-u_0$ for initial condition \eqref{perturbedIC} and $\tilde{L}=0.05$. }
\label{Interfacelost}
\end{figure}

%

\section{Biaxial boundary conditions on a disc}
\label{sec:2D}
The Dirichlet conditions in Sections~\ref{sec:RH} and \ref{sec:cylinder} are uniaxial minima of the bulk potential, $f_B$ and are referred to as minimal boundary conditions. In this section, we employ a biaxial planar boundary condition that is not a minimum of $f_B$ at the nematic-isotropic transition temperature, given by
\begin{equation}
\label{eq:b1new}
\Qvec_b = \frac{B}{3C}\left(\nvec_1\otimes \nvec_1 - \mvec\otimes \mvec \right).
\end{equation}
 This boundary condition is maximally biaxial with a zero eigenvalue i.e. $\textrm{tr} \Qvec_b^3 = 0$. We refer to $\Qvec_b$ as a non-minimal Dirichlet condition. From \cite{shafrir}, we expect that LdG energy minimizers, subject to a boundary condition of this form, converge in an appropriate sense, to a limiting harmonic map that is a minimum of the bulk potential almost everywhere and develop a boundary layer near $r=1$ to match $\Qvec_b$, in the vanishing elastic constant limit.

We study two dimensional (2D) and 3D dynamic solutions separately. A 2D solution is a symmetric and traceless $2\times 2$ matrix \cite{pre2012} and in such cases, we study maps from the disc to a 2D target space, with just two degrees of freedom. A 3D solution is a symmetric, traceless $3\times 3$ matrix and in such cases, we study maps from a 2D domain into a five-dimensional target space.

We start this section with a discussion of the 2D case; 2D $\Qvec$-matrices have $\textrm{tr}\Qvec^3=0$ and the corresponding evolution law simplifies to
\begin{equation}
\label{eq:b1b}
\Qvec_t = L \Delta \Qvec - A \Qvec - C |\Qvec|^2 \Qvec.
\end{equation}
The simplest 2D dynamic solution, consistent with (\ref{eq:b1new}), is
\begin{equation}
\label{eq:b2}
\Qvec(\rvec, t) = s(r,t) \left(\nvec_1\otimes \nvec_1 - \mvec\otimes \mvec \right)
\end{equation}
with $r^2 = x^2 + y^2$. It is simple to check that the gradient flow model (\ref{eq:b1b}) admits a solution of the form (\ref{eq:b2}) if the function $s(r,t)$ is a solution of
\begin{equation}
\label{eq:b3} \gamma s_t = \left\{s_{rr} + \frac{s_r}{r} -
\frac{4s}{r^2} \right\} - \frac{s}{L}\left( \frac{B^2}{27 C} + 2C
s^2 \right)
\end{equation}
with fixed boundary conditions
\begin{equation}
\label{eq:b4}
s(0, t) = 0 \quad s(1, t)=\frac{B}{3C}
\end{equation} for all $t\geq 0$. 
The evolution equation (\ref{eq:b3}) is simply the gradient flow model for the functional
\begin{equation}
\label{eq:b4b}
I[s]: = \int_{0}^{1}\left[ r \left(\frac{ds}{dr}\right)^2 + \frac{ 2s^2}{r} \right] + \frac{r}{L}\left(\frac{B^2}{27 C} s^2 + \frac{C}{4}s^4 \right)~\mathrm{d}r
\end{equation} 
and given a smooth solution, $s(r,t)$ of (\ref{eq:b3})-(\ref{eq:b4}) with suitable initial conditions, (\ref{eq:b2}) is the unique 2D solution. Further, the 2D potential has an isolated minimum at $s=0$ (see (\ref{eq:b4b})) and hence we expect that any dynamic solution of (\ref{eq:b1b}) has an outward-propagating interface that separates an almost isotropic core around $r=0$ from the Dirichlet boundary condition at $r=1$.
The interface equilibrates near $r=1$, followed by a sharp boundary layer to match the fixed boundary condition.

Next, we present some heuristics for 3D dynamic solutions with a planar initial condition of the
form
\begin{equation}
\label{eq:b4c}
\Qvec(\rvec, 0) = s\left(r, 0 \right) \left(\nvec_1\otimes \nvec_1 - \mvec\otimes \mvec \right),
\end{equation}
where $s(r,0)$ has an interface structure i.e.
\begin{equation}
s(r, 0) = \begin{cases} 0 & r < r_0 \\
\frac{B}{3C} & r_0 < r \leq 1
\end{cases}
\end{equation} 
for some $r_0 \in \left(0, 1 \right)$. Based on the
analysis in Section~\ref{sec:uv}, all dynamic solutions remain planar with $Q_{13} = Q_{23}=0$. Hence, we expect $\Qvec(\rvec, t)$ (for long
times) to have an isotropic core at $r=0$ and  $\Qvec(\rvec, t)$ converges to $\Qvec_1$ (defined in
(\ref{eq:c8})) everywhere away from $r=0$, with a boundary layer near $r=1$ to match the
Dirichlet condition. However, we speculate that there is a second scenario for non-minimal boundary conditions as in (\ref{eq:b1new}) which is not observed for minimal boundary conditions as in Section~\ref{sec:cylinder}. If $1 - r_0$ is sufficiently small, 3D solutions may exhibit an outward growing isotropic core since the isotropic phase is also a minimizer of the bulk
potential and this scenario may be energetically favourable.
All dynamic solutions with non-minimal boundary conditions develop a boundary layer near $r=1$, which has an energetic cost. However, solutions with minimal boundary conditions (as in Sections~\ref{sec:RH} and \ref{sec:cylinder}) do not have boundary layers near $r=1$ and in such cases, it is energetically preferable to either have a localized core of reduced order near $r=0$ (as for planar 3D solutions)
or to have uniform order throughout the disc (as for non-planar 3D solutions). 

Thus, we expect sharp contrast in the behaviour of 2D and 3D
solutions. The 2D dynamic solutions have little
nematic order being largely isotropic or close to isotropic, except near $r=1$
and 3D dynamic solutions are largely uniaxial
with perfect nematic ordering (for at least a range of values of
$r_0$), except near $r=0$ and $r=1$.  It is unknown whether 2D dynamic solutions could be
physically relevant for severely confined cylinders and if so,
what happens at the critical transition between 2D and 3D
dynamic solutions. In the next section, we present numerical
results to validate our heuristics above.

\subsection{Numerical simulations with biaixial boundary conditions}
We numerically solve (\ref{eq:n1}-\ref{eq:n5}) on a disc with the Dirichlet condition (\ref{eq:b1new}) on $r=1$. All other parameter values are as in Section~\ref{sec:dynamic}, with $\tilde{L}=0.01$.
The planar initial condition is as in \eqref{eq:b4c} 
with $s(r)=\frac{h_+}{2}\left(1+\tanh \left(\frac{r-r_0}{\sqrt{\tilde{L}}} \right)\right)$.
In the case of $r_0=0.5$, the solution quickly becomes almost uniaxial in the interior by developing an inwards-propagating well-defined isotropic-nematic interface. This is illustrated by a plot of eigenvalue evolution in Figure \ref{BiaxialityCL}. The solution converges to $\Qvec_1(\rvec, t)$ (see Equation(\ref{eq:c8})) away from $r=0$ for long times, modulo the core of reduced order near $r=0$ and 
a thin boundary layer near $r=1$, as displayed in  Figure \ref{Q1ComparisonCLR}. This is as expected from the numerical results presented in Section~\ref{sec:cylinder}.

Next, we consider $r_0=0.92$ and observe a different behaviour; the interface evolves so there is a thin boundary layer near $r=1$ with a large almost isotropic core in the interior. This is best illustrated with radial profiles of $|\boldsymbol{Q}|^2$ as seen in Figure \ref{RadialprofileCL}.
%
\begin{figure}
\centering
\centerline{\includegraphics[width=4.7cm]{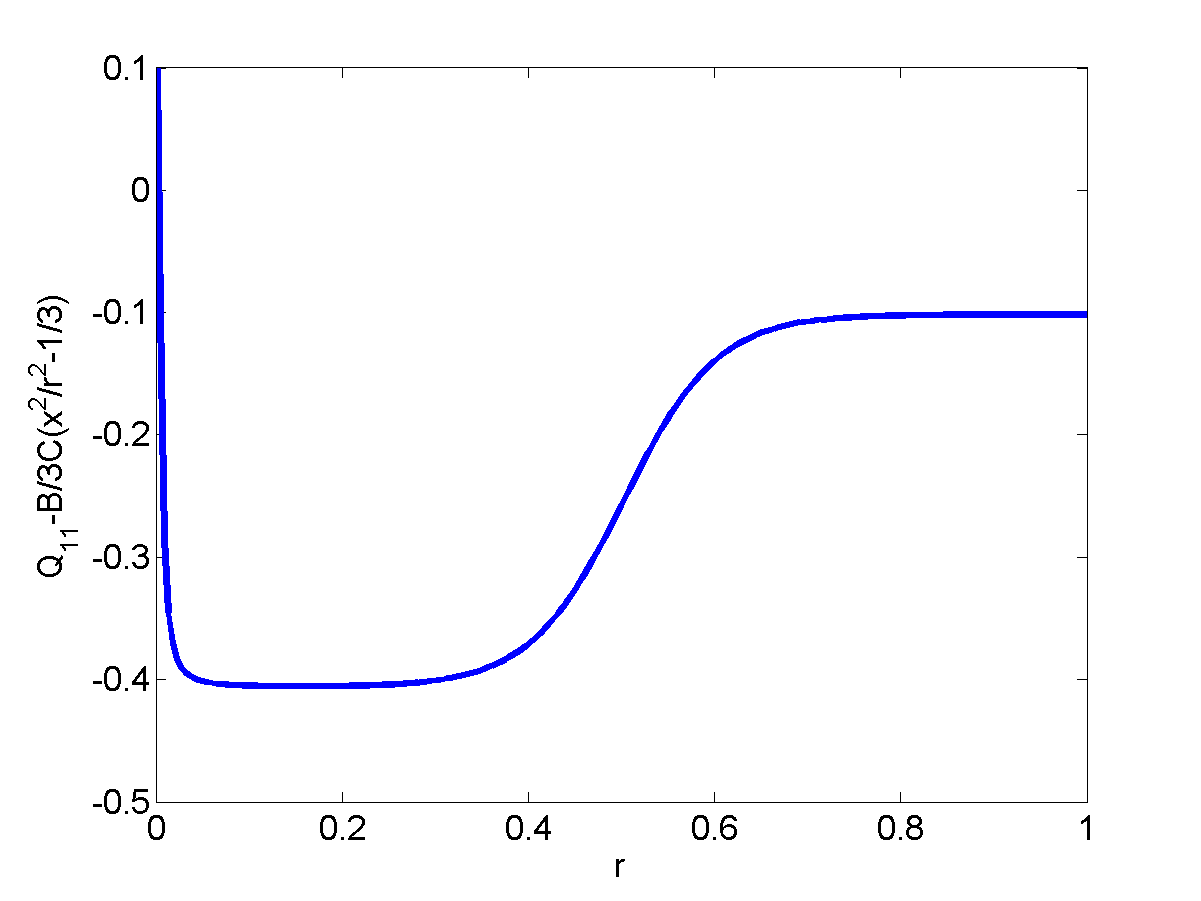}\includegraphics[width=4.7cm]{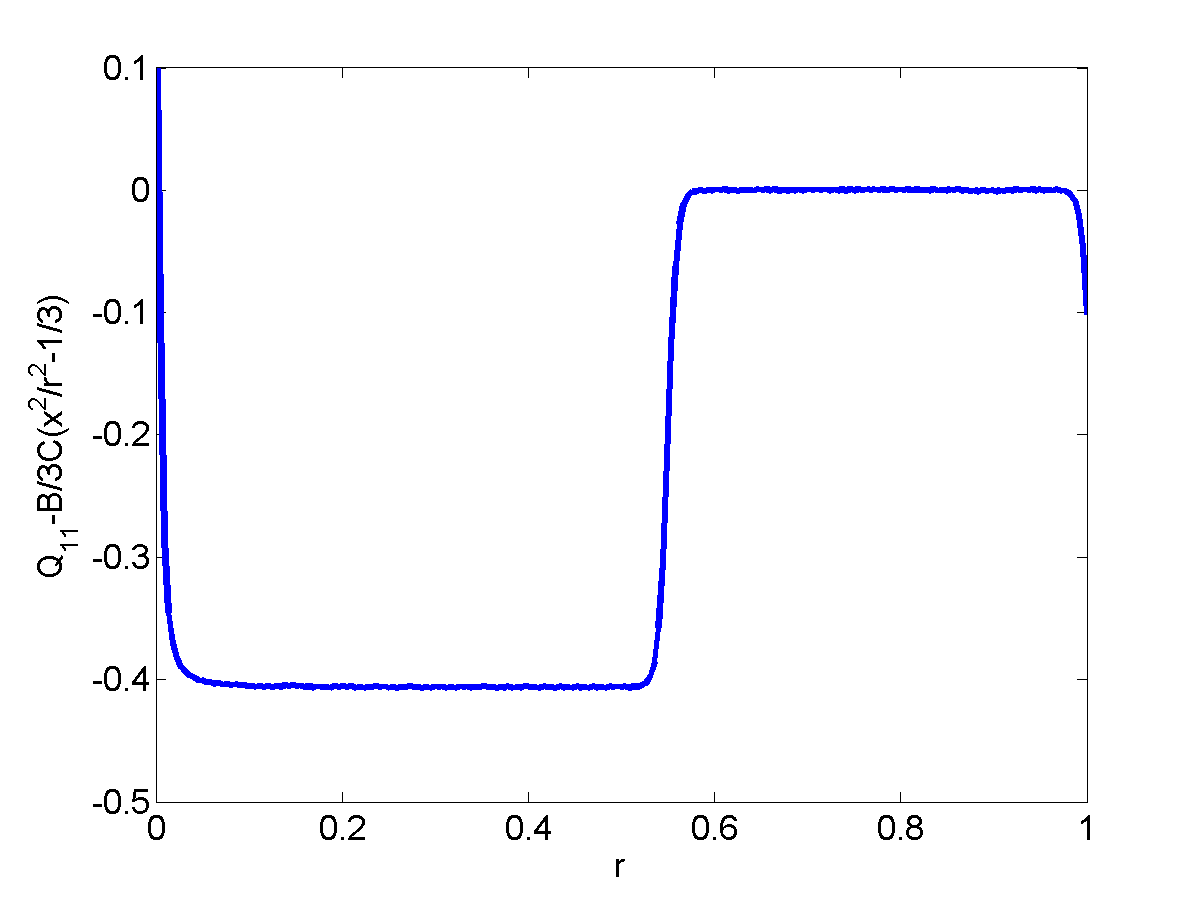}\includegraphics[width=4.7cm]{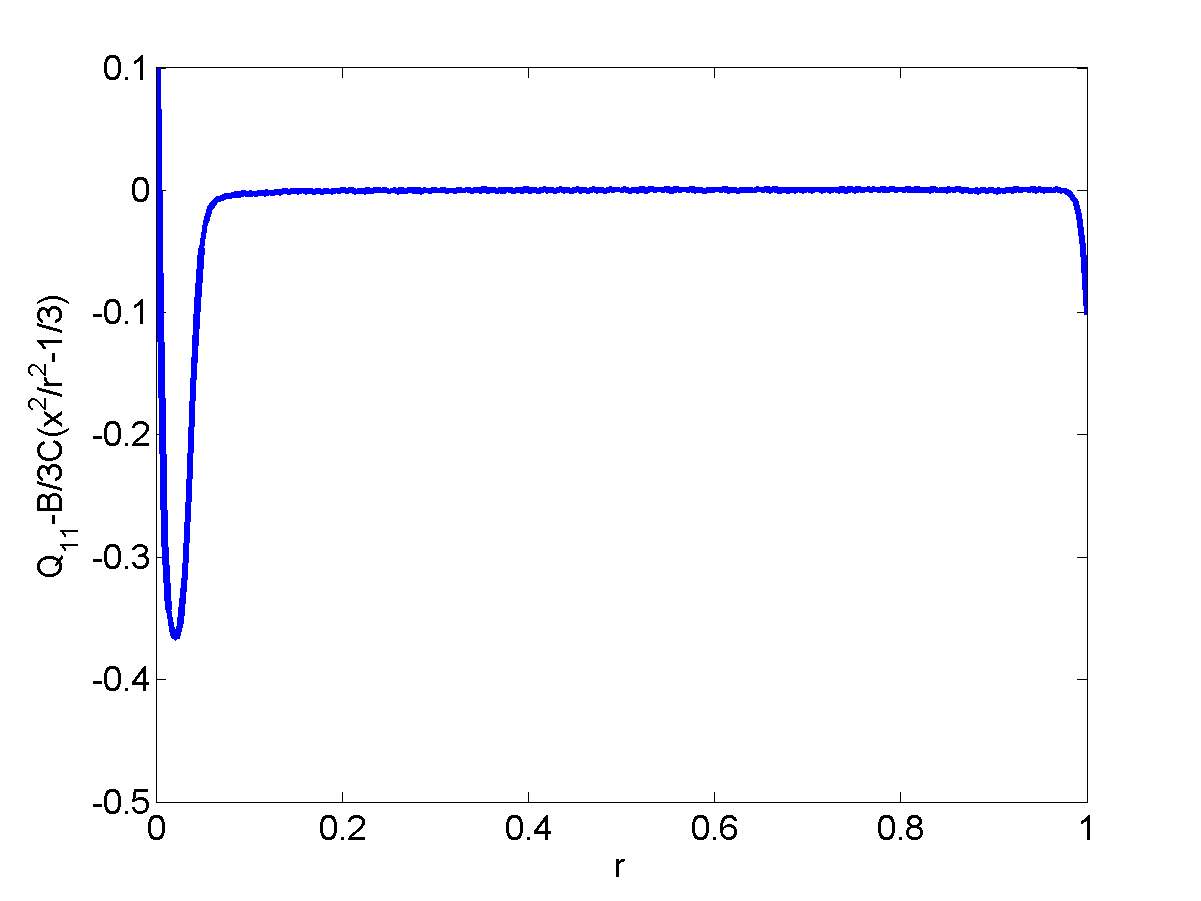}}
\caption{Radial profile of $Q_{11}(\rvec,t)-\frac{B}{3C}(\frac{x^2}{r^2}-1/3)$ for $\theta=0$ for initial
condition \eqref{eq:b4c} (with $r_0 = 0.5$), at $t=0$, $t=0.001$ and $t=0.25$.}
\label{Q1ComparisonCLR}
\end{figure}
\begin{figure}
\centering
\centerline{\includegraphics[width=4.7cm]{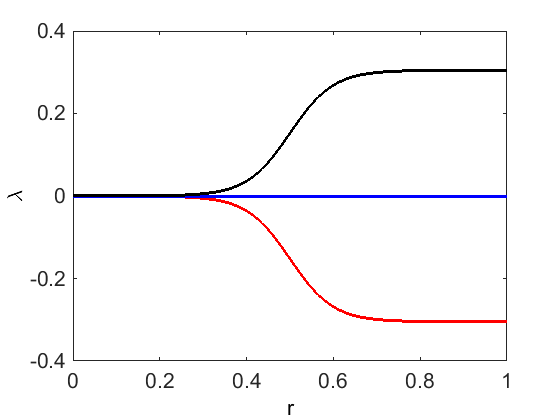}\includegraphics[width=4.7cm]{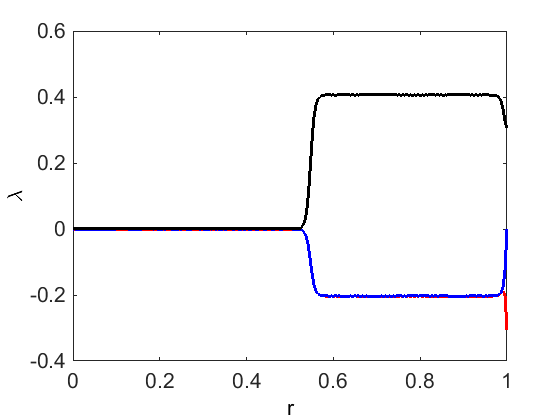}\includegraphics[width=4.7cm]{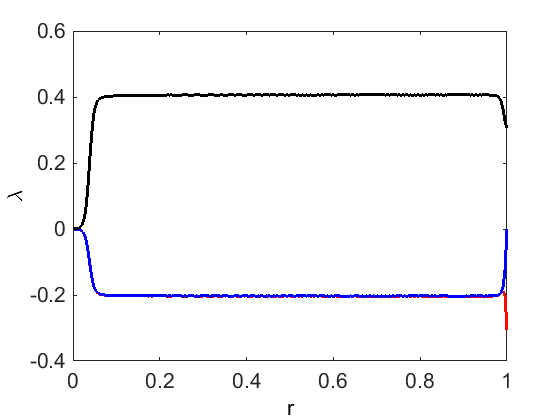}}
\caption{Eigenvalues of $\boldsymbol{Q}(\rvec,t)$ for initial
condition \eqref{eq:b4c} (with $r_0 = 0.5$), at $t=0$, $t=0.001$ and $t=0.25$.}
\label{BiaxialityCL}
\end{figure}
\begin{figure}
\centering
\centerline{\includegraphics[width=5cm]{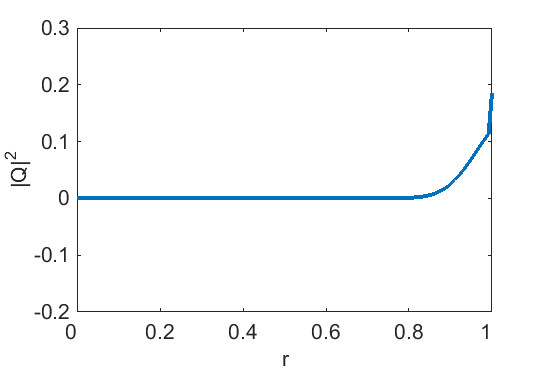}\includegraphics[width=5cm]{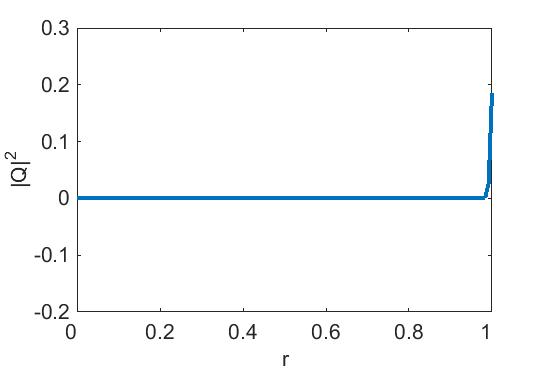}}
\caption{Radial profile of $|\boldsymbol{Q}(\rvec,t)|^2$ for initial
condition \eqref{eq:b4c} with $r_0=0.92$, for $t=0$ and $t=0.25$.}
\label{RadialprofileCL}
\end{figure}

We compare the 3D solutions above with 2D solutions for
the same system (\ref{eq:b1b}). We work with planar initial conditions \eqref{eq:b4c},
with $s(r)$ as before. The initial interface grows rapidly to yield an
almost entirely isotropic interior with a thin boundary layer near
$r=1$. This is illustrated by the eigenvalue evolution in Figure
\ref{DiscEvals}, thus corroborating our heuristics and analytical reasoning in the previous section.
\begin{figure}
\centering
\centerline{\includegraphics[width=5cm]{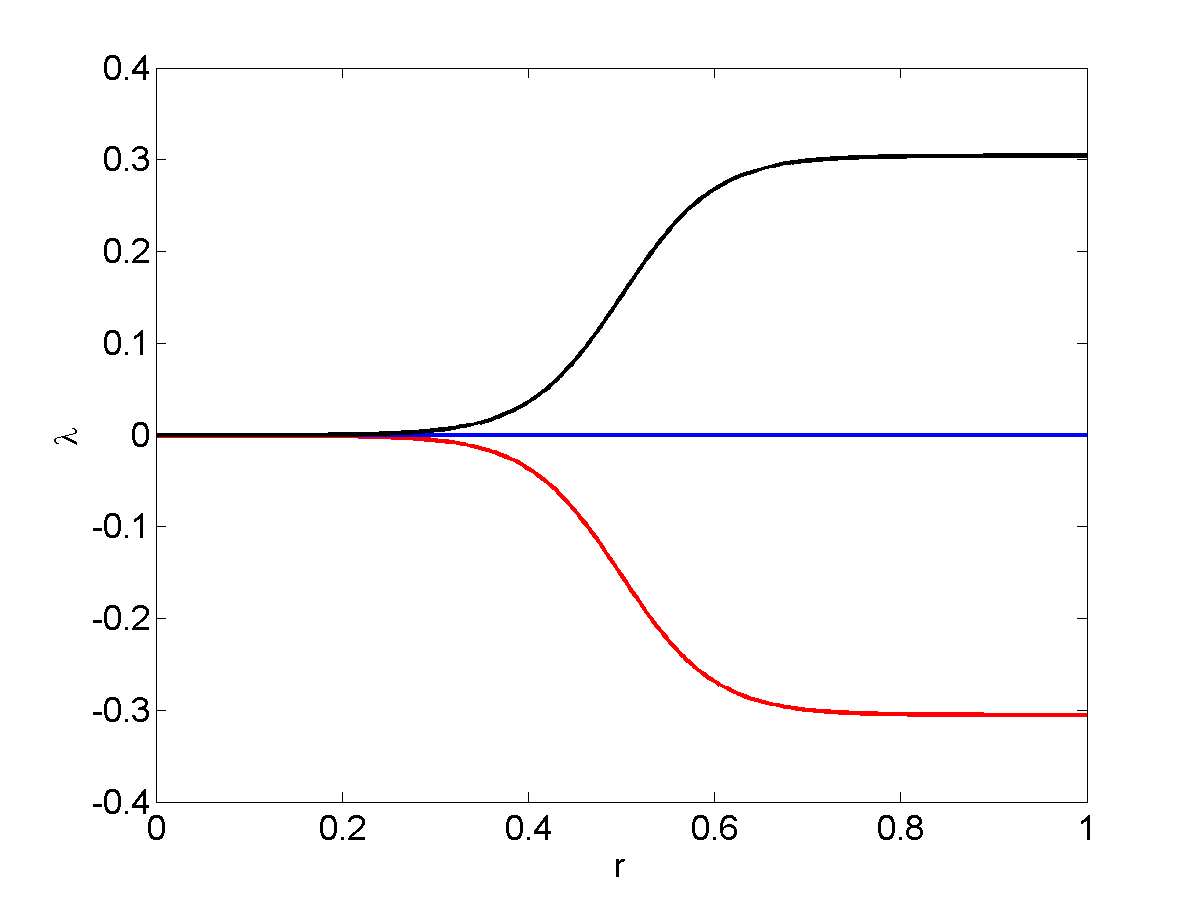}\includegraphics[width=5cm]{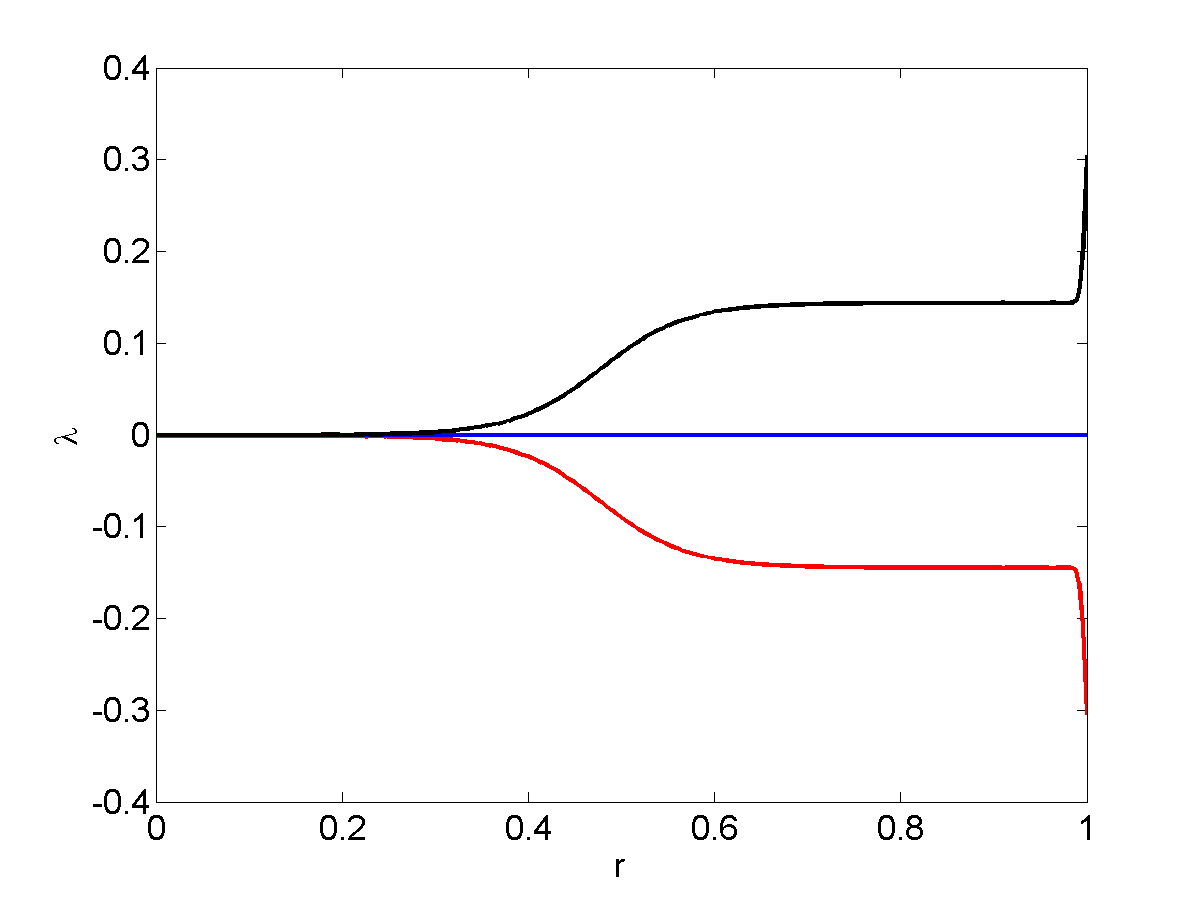}\includegraphics[width=5cm]{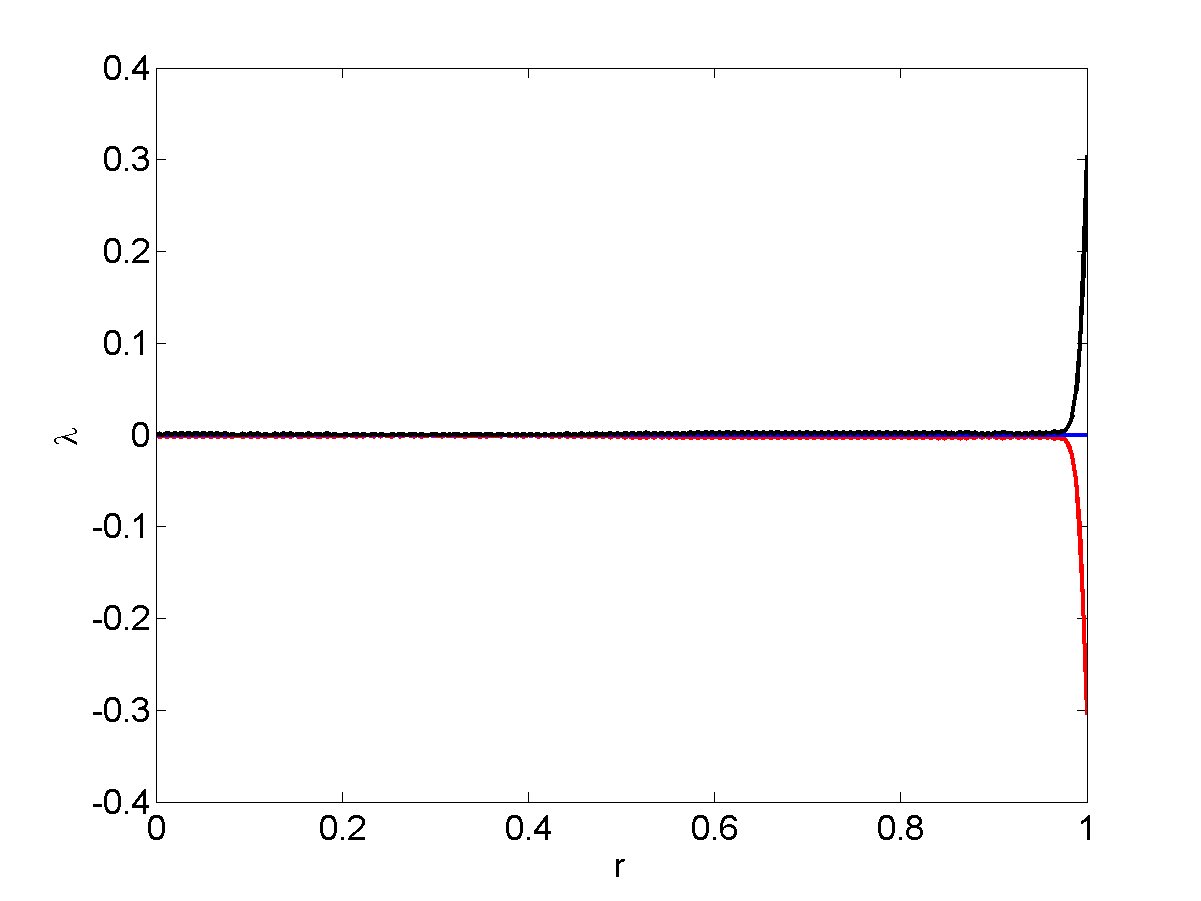}}
\caption{Eigenvalues of $\Qvec(\rvec,t)$ in the 2D model, for initial condition
\eqref{eq:b4c}, at $t=0$, $t=2 \times 10^{-5}$ and $t=2 \times10^{-4}$.} \label{DiscEvals}
\end{figure}
\section{Conclusions}
\label{sec:conclusions}
We focus on the gradient flow model for the LdG energy on prototype geometries, such as a droplet and a disc, with Dirichlet boundary conditions and various initial conditions at the nematic-isotropic transition temperature. In Section~\ref{sec:RH}, we consider the model problem of a 3D droplet of radius $R$, with radial boundary conditions. In the case of uniaxial radially symmetric initial conditions with an isotropic-nematic front structure, we adapt Ginzburg-Landau methods from \cite{bronsardkohn, bronsardstoth} to prove that the isotropic-nematic interface propagates according to mean curvature in the $\frac{L}{R^2}\frac{C}{B^2} \to 0$ limit. However, the qualitative dynamics seem universal for a large class of radially symmetric and non-symmetric uniaxial and biaxial initial conditions and the long-time dynamics are determined by the classical RH solution, which has been numerically demonstrated to be a global LdG energy minimizer in this regime.

In Sections~\ref{sec:cylinder} and \ref{sec:2D}, we focus on dynamic solutions on a disc.
Our results are largely numerical and complemented
by heuristics and analytical reasoning. We demonstrate how a choice of planar or non planar initial condition can influence the long-time dynamic behaviour. Planar initial conditions generate planar dynamic solutions with an isotropic core at the centre for all times whereas non-planar solutions follow the planar dynamics for a length of time, before 
relaxing into an uniaxial state of perfect order for long times. In Section~\ref{sec:2D}, we look at 
non-minimal boundary conditions. Non-minimal boundary conditions allow for dynamic scenarios outside the scope of minimal boundary conditions and since minimal boundary conditions are an idealization, non-minimal Dirichlet conditions can be physically relevant too. 

The long-time dynamics can be understood in terms of local and global minimizers, or in some cases critical points, of the LdG energy. 
In cases where the LdG critical points exhibit an isotropic-nematic interface, this interface may be localized with little effect on global properties. Our numerical results show that a large class of physically relevant LC model problems can exhibit a well-defined isotropic-nematic interface for a length of time (see Figures~\ref{InterfaceEvolution} and \ref{uvinterface}) and these results give insight into how boundary and initial conditions can be used to yield either largely disordered or ordered nematic profiles. A natural next step is to rigorously analyze front formation and propagation with generic non-minimal boundary conditions and with more general LdG energy functionals, including those with a sixth order bulk potential that allow for biaxial minima. 
We will report on these developments in the future.

\section{Acknowledgements}
A.M. is supported by an EPSRC Career Acceleration Fellowship EP/J001686/1 and EP/J001686/2 and an OCIAM Visiting Fellowship. P.A.M. gratefully acknowledges support from a Royal Society Wolfson award. A.S. is supported by an Engineering and Physical Sciences Research Council (EPSRC) studentship. We thank Heiko Gimperlein and Giacomo Canevari for helpful comments.
\bibstyle{plain}
\bibliography{frontsv19}
\end{document}